\documentclass[a4paper]{amsart}
\usepackage[latin1]{inputenc}
\usepackage{amsmath}
\usepackage{amsthm}
\usepackage{amsfonts}
\usepackage{amssymb}
\usepackage[T1]{fontenc}
\usepackage{float}
\usepackage[all]{xy}
\usepackage{tikz}
\usetikzlibrary{arrows}
\usepackage{hyperref}
\usepackage{calrsfs}
\usepackage{bbm}
\usepackage{color}
\usepackage[hmargin=3cm,vmargin=4.4cm]{geometry}
\usepackage{tabularx}
\usepackage{bbm}

\newtheorem{theorem}{Theorem}[section]
\newtheorem{prop}[theorem]{Proposition}
\newtheorem{lem}[theorem]{Lemma}
\newtheorem{corol}[theorem]{Corollary}

\newtheorem{prob}[theorem]{Problem}

\theoremstyle{definition}
\newtheorem{defi}[theorem]{Definition}

\newtheorem{rmq}[theorem]{Remark}
\newtheorem{exmp}[theorem]{Example}

\def\Gr{{\rm{Gr}}}
\def\pr{{\rm{pr}}}

\def\Mut{{\rm{Mut}\,}}
\def\dim{{\rm{dim}\,}}

\def\ens#1{\left\{ #1 \right\}}
\def\fl{{\longrightarrow}\,}

\def\CC{{\mathcal{C}}}

\def\C{{\mathbb{C}}}

\def\Q{{\mathbb{Q}}}

\def\Z{{\mathbb{Z}}}

\def\id{\1}
\def\op{{\rm{op}}}
\def\ens#1{\left\{ #1 \right\}}
\def\Ext{{\rm{Ext}}}

\def\End{{\rm{End}}}
\def\Hom{{\rm{Hom}}}

\def\d{{\partial}}

\def\im{{\rm{Im}}}
\def\ker{{\rm{Ker}}}

\def\res{{\rm{Res}}}

\def\1{\mathbbm{1}}

\def\k{{\mathbbm{k}}}
\def\x{{\mathbf x}}
\def\ex{{\mathbf{ex}}}
\def\y{{\mathbf y}}

\def\am{\mathbf{Clus}}
\def\amr{\am}

\def\Ann{{\mathbf{Ring}}}

\def\AM#1{{\bf{(CM#1)}}}

\title{On a category of cluster algebras}
\author{Ibrahim Assem, Gr\'egoire Dupont and Ralf Schiffler}
\address{Universit\'e de Sherbrooke, Sherbrooke QC, Canada
}
\email{ibrahim.assem@usherbrooke.ca}
\address{Institut de Math\'ematiques de Jussieu-Paris Rive Gauche, Paris, France
}
\email{dupontg@math.jussieu.fr}
\address{University of Connecticut, Storrs CT, USA
}
\email{schiffler@math.uconn.edu}

\date{\today}

\begin{document}

\begin{abstract}
	We introduce a category of cluster algebras with fixed initial seeds. This category has countable coproducts, which can be constructed combinatorially, but no products. We characterise isomorphisms and monomorphisms in this category and provide combinatorial methods for constructing special classes of monomorphisms and epimorphisms. In the case of cluster algebras from surfaces, we describe interactions between this category and the geometry of the surfaces.
\end{abstract}

\maketitle

\setcounter{tocdepth}{1}
\tableofcontents

\section*{Introduction}
	Cluster algebras are particular commutative algebras which were introduced by Fomin and Zelevinsky in \cite{cluster1} in order to provide a combinatorial framework for studying total positivity and canonical bases in algebraic groups. Since then, a fast-growing literature focused on the numerous interactions of these algebras with various areas of mathematics like Lie theory, Poisson geometry, representation theory of algebras or mathematical physics. The study of the cluster algebras as algebraic structures in themselves can essentially be found in the seminal series of articles \cite{cluster1,cluster2,cluster3,cluster4} and their ring-theoretic properties were recently studied in \cite{GLS:factorial}. 

	For a long time, an obstacle to the good understanding of cluster algebras was that they are defined \emph{recursively} by applying a combinatorial process called \emph{mutation}. However, the interactions of cluster algebras with either the (combinatorial) Teichmüller theory or the representation theory of algebras led to various closed formulae which enlightened the structure of cluster algebras, see \cite{MSW:positivity, DWZ:potentials2,Plamondon:ClusterAlgebras,HL:quantumaffinecluster}.

	In order to get a better comprehension of cluster algebras, the next step is thus to define a categorical framework for their study. Therefore, one has to find what are the ``right'' morphisms between cluster algebras. The most natural idea is to look at ring homomorphisms which commute with mutations. For bijective morphisms from a coefficient-free skew-symmetric cluster algebra to itself, this gave rise to the notion of \emph{cluster automorphisms}, introduced in \cite{ASS:automorphisms}. However, in more general settings this idea turns out to be slightly too restrictive in order to get enough morphisms between non-isomorphic cluster algebras.

	In this article, we slightly relax this first idea and propose a similar definition of cluster morphism between arbitrary skew-symmetrisable cluster algebras of geometric type (with non-invertible coefficients). The definition is still based on the idea that morphisms between cluster algebras should commute with mutations but, broadly speaking, we also allow the morphisms to specialise certain cluster variables to integral values or to send frozen variables to exchangeable ones (see Definition \ref{defi:main}).

	With this notion of morphisms, we obtain a category $\am$ with countable coproducts (Lemma \ref{lem:sommes}) but generally without products (Proposition \ref{prop:products}). We prove that  in $\am$, the isomorphisms are the bijective morphisms (Corollary \ref{corol:iso}), the monomorphisms are the injective monomorphisms (Proposition \ref{prop:mono}) while the epimorphisms are not necessarily surjective (Remark \ref{rmq:cexepi}). 

	Inspired by the interactions between geometry and the combinatorics of cluster algebras associated with surfaces in the sense of \cite{FST:surfaces}, we define for arbitrary cluster algebras concepts of \emph{gluings} and \emph{cuttings} which provide natural classes of monomorphisms and epimorphisms in $\am$ (see Sections \ref{section:sums} and \ref{section:surgery}). 

	We also study specialisations of cluster variables in this categorical context. We prove that the usual specialisations of \emph{frozen} variables to 1 yield epimorphisms in $\am$. More surprisingly, for cluster algebras from surfaces or for acyclic cluster algebras, we prove that specialisations of exchangeable cluster variables also give rise to epimorphisms in $\am$ (Theorems \ref{theorem:specialisationsurfaces} and \ref{theorem:acyclicspe}). 

 	The article is organised as follows. After a brief section recalling our conventions, Section \ref{section:rootedclalg} recalls the principal definitions on cluster algebras which we will use along the article. In Section \ref{section:amr}, we introduce the notion of \emph{rooted cluster morphisms} and the category $\amr$ and we establish some basic properties. Section \ref{section:iso} is devoted to the study of isomorphisms in $\amr$, generalising previous results of \cite{ASS:automorphisms} on cluster automorphisms. Section \ref{section:mono} is devoted to the study of monomorphisms in $\amr$. Section \ref{section:sums} concerns the study of products and coproducts in $\amr$ and their interactions with the geometry of surfaces. Section \ref{section:epi} is devoted to the study of epimorphisms in $\amr$ and gives rise to a combinatorial process called \emph{surgery}, whose interactions with the geometry of surfaces are studied in Section \ref{section:surgery}.

\section*{Notations}\label{section:notations}
	In this article, every ring $A$ has an identity $1_A$ and every ring homomorphism $A \fl B$ is assumed to send the identity $1_A$ to the identity $1_B$. We denote by $\Ann$ the category of rings with ring homomorphisms. 

	If $I$ is a countable set, we denote by $M_I(\Z)$ the ring of locally finite matrices with integral entries indexed by $I \times I$ (we recall that a matrix $B = (b_{ij})_{i,j \in I}$ is \emph{locally finite} if for every $i \in I$, the families $(b_{ij})_{j \in I}$ and $(b_{ji})_{j \in I}$ have finite support). We say that $B$ is \emph{skew-symmetrisable} if there exists a family of non-negative integers $(d_i)_{i \in I}$ such that $d_i b_{ij} = - d_j b_{ji}$ for any $i,j \in I$. If $J \subset I$, we denote by $B[J]=(b_{ij})_{i,j \in J}$ the submatrix of $B$ formed by the entries labelled by $J \times J$.

	If $I$ and $J$ are sets, we use the notation $I \setminus J = \ens{i \in I \ | \ i \not \in J}$ independently of whether $J$ is contained in $I$ or not. By a \emph{countable} set we mean a set of cardinality at most $\aleph_0$.

	If $R$ is a subring of a ring $S$, and if $\mathcal X \subset S$ is a set, we denote by $R[\mathcal X]$ the ring of all polynomials with coefficients in $R$ evaluated in the elements of $S$. Note that this is not necessarily isomorphic to a ring of polynomials.

	We recall that a \emph{concrete} category is a category whose objects have underlying sets and whose morphisms between objects induce maps between the corresponding sets.

	In order to make some statements clearer, it might be convenient for the reader to use a combinatorial representation of the skew-symmetrisable locally finite matrices as \emph{valued quivers}. If $B \in M_I(\Z)$ is a locally finite skew-symmetric matrix, we associate with $B$ a valued quiver $Q_B$ whose points are indexed by $I$ and such that for any $i,j \in I$, if $b_{ij}>0$ (so that $b_{ji}<0$), then we draw a valued arrow 
	$$\xymatrix{ i \ar[rr]^{(b_{ij},-b_{ji})} && j}$$
	in $Q_B$. As $B$ is skew-symmetrisable the valued quiver $Q_B$ has no oriented cycles of length at most two. In the case where $B$ is skew-symmetric, if $i,j \in I$ are such that $b_{ij}>0$, we usually draw $b_{ij}$ arrows from $i$ to $j$ in $Q_B$ instead of a unique arrow with valuation $(b_{ij},-b_{ji})$.

\section{Rooted cluster algebras}\label{section:rootedclalg}
	In this section we recall the definition of a cluster algebra of geometric type. As opposed to the initial definition formulated in \cite{cluster1}, we consider \emph{non-invertible} coefficients.

	\subsection{Seeds and mutations}
		\begin{defi}[Seeds]
			A \emph{seed} is a triple $\Sigma = (\x, \ex, B)$ such that~:
			\begin{enumerate}
				\item $\x$ is a countable set of undeterminates over $\Z$, called the \emph{cluster} of $\Sigma$~;
				\item $\ex \subset \x$ is a subset of $\x$ whose elements are the \emph{exchangeable variables} of $\Sigma$~;
				\item $B =(b_{xy})_{x,y \in \x} \in M_{\x}(\Z)$ is a (locally finite) skew-symmetrisable matrix, called the \emph{exchange matrix} of $\Sigma$.
			\end{enumerate}
		\end{defi}

		The variables in $\x \setminus \ex$ are the \emph{frozen variables} of $\Sigma$. A seed $\Sigma=(\x,\ex,B)$ is \emph{coefficient-free} if $\ex=\x$ and in this case, we simply write $\Sigma=(\x,B)$. A seed is \emph{finite} if $\x$ is a finite set. 

		Given a seed $\Sigma$, the field $\mathcal F_{\Sigma} = \Q(x\ | \ x \in \x)$ is called the associated \emph{ambient field}.

		\begin{defi}[Mutation]
			Given a seed $\Sigma = (\x, \ex, B)$ and an exchangeable variable $x \in \ex$, the image of the \emph{mutation} of $\Sigma$ in the direction $x$ is the seed $$\mu_x(\Sigma) = (\x',\ex',B')$$ given by
			\begin{enumerate}
				\item $\x' = (\x \setminus \ens{x}) \sqcup \ens{x'}$ where
					$$xx' = \prod_{\substack{y \in \x~; \\ b_{xy}>0}} y^{b_{xy}} + \prod_{\substack{y \in \x~; \\ b_{xy}<0}} y^{-b_{xy}}.$$
				\item $\ex'=(\ex \setminus \ens{x}) \sqcup \ens{x'}$.
				\item $B'=(b'_{yz}) \in M_{\x}(\Z)$ is given by
					$$b'_{yz} = \left\{\begin{array}{ll}
						- b_{yz} & \textrm{ if } x=y \textrm{ or } x=z~; \\
						b_{yz} + \frac 12 (|b_{yx}|b_{xz} + b_{yx}|b_{xz}|) & \textrm{ otherwise.}
					\end{array}\right.$$
			\end{enumerate}
		\end{defi}
		
		For any $y \in \x$ we denote by $\mu_{x,\Sigma}(y)$ the variable corresponding to $y$ in the cluster of the seed $\mu_x(\Sigma)$, that is, $\mu_{x,\Sigma}(y)=y$ if $y \neq x$ and $\mu_{x,\Sigma}(x)=x'$ where $x'$ is defined as above. If there is no risk of confusion, we simply write $\mu_x(y)$ instead of $\mu_{x,\Sigma}(y)$.

		The set $\x'$ is again a free generating set of $\mathcal F_{\Sigma}$ and the mutation is involutive in the sense that $\mu_{x'} \circ \mu_x(\Sigma)=\Sigma$, for each $x \in \ex$.

		\begin{defi}[Admissible sequence of variables]
			Let $\Sigma=(\x,\ex,B)$ be a seed. We say that $(x_1, \ldots, x_l)$ is \emph{$\Sigma$-admissible} if $x_1$ is exchangeable in $\Sigma$ and if, for every $i \geq 2$, the variable $x_i$ is exchangeable in $\mu_{x_{i-1}} \circ \cdots \circ \mu_{x_1}(\Sigma)$.
		\end{defi}

		Given a seed $\Sigma=(\x,\ex,B)$, its \emph{mutation class} is the set $\Mut(\Sigma)$ of all seeds which can be obtained from $\Sigma$ by applying successive mutations along finite admissible sequences of variables. In other words, 
		$$\Mut(\Sigma) = \ens{\mu_{x_{n}} \circ \cdots \circ \mu_{x_1}(\Sigma) \ | \ n>0 \text{ and } (x_1, \ldots, x_n) \textrm{ is $\Sigma$-admissible}}.$$
		Two seeds in the same mutation class are \emph{mutation-equivalent}.

	\subsection{Rooted cluster algebras}
		\begin{defi}[Rooted cluster algebra]
			Let $\Sigma$ be a seed. The \emph{rooted cluster algebra with initial seed} $\Sigma$ is the pair $(\Sigma,\mathcal A)$ where $\mathcal A$ is the $\Z$-subalgebra of $\mathcal F_{\Sigma}$ given by~:
			$$\mathcal A = \mathcal A(\Sigma) = \Z\left[x \ | \ x \in \bigcup_{(\x,\ex,B) \in \Mut(\Sigma)} \x\right] \subset \mathcal F_{\Sigma}.$$
			The variables (exchangeable variables and frozen variables respectively) arising in the clusters of seeds which are mutation-equivalent to $\Sigma$ are the \emph{cluster variables} (or the \emph{exchangeable variables} and \emph{frozen variables} respectively) of the rooted cluster algebra $(\Sigma, \mathcal A)$. We denote by $\mathcal X_\Sigma$ the set of cluster variables in $(\Sigma,\mathcal A)$.
		\end{defi}

		In order to simplify notations, a rooted cluster algebra $(\Sigma,\mathcal A)$ is in general simply denoted by $\mathcal A(\Sigma)$ but one should keep in mind that a rooted cluster algebra is always viewed together \emph{with} its initial seed.

		\begin{rmq}
			This definition authorises seeds whose clusters are empty. Such seeds are called \emph{empty seeds} and by convention the rooted cluster algebra corresponding to an empty seed is $\Z$.
		\end{rmq}

		\begin{exmp}
			If $\Sigma=(\x,\emptyset,B)$ has no exchangeable variables, then $$\mathcal A(\Sigma) = \Z[x \ | \ x \in \x]$$ is a polynomial ring in countably many variables.
		\end{exmp}

		\begin{rmq}
			In the original definition of cluster algebras given in \cite{cluster1}, the frozen variables are supposed to be invertible in the cluster algebra. However, it is known that cluster variables in a cluster algebra of geometric type are Laurent polynomials in the exchangeable variables with polynomial coefficients in the frozen ones (see for instance \cite[Proposition 11.2]{cluster2}). Therefore, the cluster algebra structure can essentially be considered without inverting the coefficients. Also, several ``natural'' examples of cluster algebras arise with non-invertible coefficients, as for instance cluster algebras arising in Lie theory as polynomial rings (see \cite[\S 6.4]{GLS:factorial}) or cluster structures on rings of homogeneous coordinates on Grassmannians (see \cite[\S 2.1]{GSV:book} or Section \ref{ssection:Gr}).

			Of course, if one wants to recover the initial definition from ours, it is enough to localise the cluster algebra at the frozen variables. Nevertheless, some of our results (in particular the crucial Lemma \ref{lem:injfrozen}) require that frozen variables be non-invertible.
		\end{rmq}

	\subsection{Rooted cluster algebras from surfaces}\label{ssection:surfaces}
		In this article, we are often interested in the particular class of rooted cluster algebras associated with marked surfaces in the sense of \cite{FST:surfaces}. We recall that a \emph{marked surface} is a pair $(S,M)$ where $S$ is an oriented 2-dimensional Riemann surface and $M$ is a finite set of marked points in the closure of $S$ such that each connected component of the boundary $\d S$ of the surface $S$ contains at least one marked point in $M$. We also assume that none of the connected components of $(S,M)$ is a \emph{degenerate marked surface}, that is, a surface which is homeomorphic to one of the following surfaces~: 
		\begin{itemize}
			\item a sphere with one, two or three punctures,
			\item an unpunctured or a once-punctured monogon,
			\item an unpunctured digon or an unpunctured triangle.
		\end{itemize}

		All curves in $(S,M)$ are considered up to isotopy with respect to the set $M$ of marked points. Therefore, two curves $\gamma$ and $\gamma'$ are called \emph{distinct} if they are not isotopic. Two curves $\gamma$ and $\gamma'$ are called \emph{compatible} if there exist representatives of their respective isotopy classes which do not intersect in $S \setminus M$.

		An \emph{arc} is a curve joining two marked points in $(S,M)$ which is compatible with itself. An arc is a \emph{boundary arc} if it is isotopic to a connected component of $\d S \setminus M$, otherwise it is \emph{internal}.

		A \emph{triangulation} of $(S,M)$ is a maximal collection of arcs which are pairwise distinct and compatible. The arcs of the triangulation cut the surface into triangles (which may be \emph{self-folded}). 

		With any triangulation $T$ of $(S,M)$ we can associate a skew-symmetric matrix $B^T$ as in \cite{FST:surfaces}. For the convenience of the reader, we recall this construction in the case where $T$ has no self-folded triangles (for the general case we refer the reader to \cite[\S 4]{FST:surfaces}). For any triangle $\Delta$ in $T$, we define a matrix $B^\Delta$, indexed by the arcs in $T$ and given by 
		$$B^\Delta_{\gamma,\gamma'} 
		= \left\{\begin{array}{ll}
			1 & \text{ if $\gamma$ and $\gamma'$ are sides of $\Delta$ and $\gamma'$ follows $\gamma$ in the positive direction~;}\\
			-1 & \text{ if $\gamma$ and $\gamma'$ are sides of $\Delta$ and $\gamma'$ follows $\gamma$ in the negative direction~;}\\
			0 & \text{ otherwise.}
		\end{array}\right.$$
		The matrix $B^T$ is then given by
		$$B^T = \sum_{\Delta} B^\Delta$$
		where $\Delta$ runs over all the triangles in $T$. 

		In terms of quivers, the quiver $Q_T$ corresponding to $B^T$ is the quiver such that~:
		\begin{itemize}
			\item the points in $Q_T$ are the arcs of $T$,
			\item the frozen points in $Q_T$ are the boundary arcs of $T$,
			\item there is an arrow $\gamma \fl \gamma'$ if and only if $\gamma$ and $\gamma'$ are sides of a same triangle and $\gamma'$ follows $\gamma$ in the positive direction,
			\item a maximal collection of 2-cycles is removed.
		\end{itemize}

		\begin{exmp}
			Figure \ref{fig:quivertriangulation} shows an example of quiver obtained from a triangulation of a hexagon. Points in white correspond to frozen variables and points in black correspond to exchangeable variables. The dashed arrows, joining frozen points, are precisely those which we remove in the \emph{simplification of the seed} (see Definition \ref{defi:simplification}).
			\begin{figure}[htb]
				\begin{center}
					\begin{tikzpicture}
						\draw[thick] (0,0) -- ++(1,1.7) -- ++(-1,1.7) -- ++ (-2,0) -- ++(-1,-1.7) -- ++(1,-1.7) -- cycle;

						\draw[] (-2,0) -- (-2,3.4);
						\draw[] (-2,0) -- (1,1.7);
						\draw[] (1,1.7) -- (-2,3.4);

						\draw[-triangle 60,red] (-1,0) -- (-.5,.85);
						\draw[-triangle 60,red] (-.5,.85) -- (.5,.85);
						\draw[-triangle 60,red,dashed] (.5,.85) -- (-1,0);

						\draw[-triangle 60,red] (-.5,.85) -- (-2,1.7);
						\draw[-triangle 60,red] (-2,1.7) -- (-.6,2.6);
						\draw[-triangle 60,red] (-.6,2.6) -- (-.5,.85);

						\draw[-triangle 60,red]  (-2.5,2.55) -- (-2,1.7);
						\draw[-triangle 60,red,dashed] (-2.5,.85) -- (-2.5,2.55);
						\draw[-triangle 60,red] (-2,1.7) -- (-2.5,.85);

						\draw[-triangle 60,red] (-.6,2.6) -- (-1,3.4);
						\draw[-triangle 60,red,dashed] (-1,3.4) -- (.5,2.55);
						\draw[-triangle 60,red] (.5,2.55) -- (-.6,2.6);

						\fill (-2,1.7) circle (.085);
						\fill (-.5,.85) circle (.085);
						\fill (-.6,2.6) circle (.085);

						\fill[white] (-1,0) circle (.085);
						\fill[white] (-1,3.4) circle (.085);
						\fill[white] (.5,.85) circle (.085);
						\fill[white] (-2.5,.85) circle (.085);
						\fill[white] (.5,2.55) circle (.085);
						\fill[white] (-2.5,2.55) circle (.085);

						\draw (-1,0) circle (.085);
						\draw (-1,3.4) circle (.085);
						\draw (.5,.85) circle (.085);
						\draw (-2.5,.85) circle (.085);
						\draw (.5,2.55) circle (.085);
						\draw (-2.5,2.55) circle (.085);
					\end{tikzpicture}
				\end{center}
				\caption{The quiver of a triangulation of a hexagon.}\label{fig:quivertriangulation}
			\end{figure}
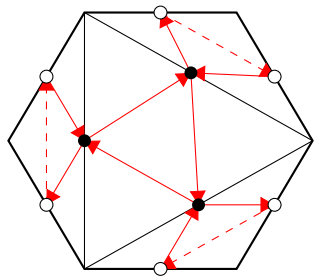
		\end{exmp}

		Then, we can associate with $T$ the \emph{seed} $\Sigma_T = (\x_T,\ex_T,B_T)$ where~:
		\begin{itemize}
			\item $\x_T$ is indexed by the arcs in $T$~;
			\item $\ex_T$ is indexed by the internal arcs in $T$~;
			\item $B^T$ is the matrix defined above.
		\end{itemize}
		The \emph{rooted cluster algebra associated with the triangulation $T$} is therefore $\mathcal A(\Sigma_T)$.

		It is proved in \cite{FST:surfaces} that if $T$ and $T'$ are two triangulations of $(S,M)$, then $\Sigma_T$ and $\Sigma_{T'}$ can be joined by a sequence of mutations. Therefore, up to a ring isomorphism, $\mathcal A(\Sigma_T)$ does not depend on the choice of the triangulation $T$ and is called \emph{the cluster algebra $\mathcal A(S,M)$ associated with $(S,M)$}, see \cite{FST:surfaces} for further details.

		More generally, it is possible to associate to a marked surface a cluster algebra with an arbitrary choice of coefficients and not only the coefficients arising from the boundary arcs considered above. However, for geometric statements (see for instance Sections \ref{ssection:typeA}, \ref{ssection:specialisationsurfaces} or \ref{ssection:topsurgery}), it is usually more natural to consider coefficients associated with the boundary arcs. 

		Finally, in order to avoid technicalities, we only consider \emph{untagged} triangulations but all the results we present can easily be extended to the case of tagged triangulations.

\section{The category of rooted cluster algebras}\label{section:amr}
	\subsection{Rooted cluster morphisms}
		\begin{defi}[Biadmissible sequences]
			Let $\Sigma = (\x,\ex,B)$ and $\Sigma'=(\x',\ex',B')$ be two seeds and let $f: \mathcal F_\Sigma \fl \mathcal F_{\Sigma'}$ be a map. A sequence $(x_1, \ldots, x_n) \subset \mathcal A(\Sigma)$ is called \emph{$(f,\Sigma,\Sigma')$-biadmissible} (or simply \emph{biadmissible} if there is no risk of confusion) if it is $\Sigma$-admissible and if $(f(x_1), \ldots, f(x_n))$ is $\Sigma'$-admissible.
		\end{defi}

		We fix two seeds $\Sigma = (\x,\ex,B)$ and $\Sigma'=(\x',\ex',B')$.

		\begin{defi}[Rooted cluster morphisms]\label{defi:main}
			A \emph{rooted cluster morphism} from $\mathcal A(\Sigma)$ to $\mathcal A(\Sigma')$ is a ring homomorphism from $\mathcal A(\Sigma)$ to $\mathcal A(\Sigma')$ such that~:
			\begin{tabularx}{15cm}{rX}
				\AM 1 & $f(\x) \subset \x' \sqcup \Z$~; \\
				\AM 2 & $f(\ex) \subset \ex' \sqcup \Z$~;\\
				\AM 3 & For every $(f,\Sigma,\Sigma')$-biadmissible sequence $(x_1, \ldots, x_n)$, we have 
				$$f(\mu_{x_n} \circ \cdots \circ \mu_{x_1,\Sigma}(y)) = \mu_{f(x_n)} \circ \cdots \circ \mu_{f(x_1),\Sigma'}(f(y))$$
				for any $y$ in $\x$. 
			\end{tabularx}
		\end{defi}
		
		We sometimes say that $f$ \emph{commutes with biadmissible mutations} if it satisfies \AM 3.

		\begin{rmq}
			A rooted cluster morphism may send a frozen cluster variable to an exchangeable cluster variable whereas \AM 2 prevents the opposite from happening.
		\end{rmq}

		\begin{rmq}
			Given an explicit ring homomorphism $f$ between two rooted cluster algebras, \AM 3 is difficult to check since it requires to test every biadmissible sequence, and there are in general infinitely many. However, we shall see that for instance for isomorphisms, it is sometimes possible to simplify this situation (see Lemma \ref{lem:amlocambij}).
		\end{rmq}

		\begin{prop}\label{prop:composition}
			The composition of rooted cluster morphisms is a rooted cluster morphism.
		\end{prop}
		\begin{proof}
			We fix three rooted cluster algebras $\mathcal A_1,\mathcal A_2$ and $\mathcal A_3$ with respective initial seeds $\Sigma_1,\Sigma_2$ and $\Sigma_3$ where $\Sigma_i=(\x_i,\ex_i,B^i)$ and consider rooted cluster morphisms $f:\mathcal A_1 \fl \mathcal A_2$ and $g:\mathcal A_2 \fl \mathcal A_3$. The composition $g \circ f$ is a ring homomorphism from $\mathcal A_1$ to $\mathcal A_3$. Moreover, we have~:
			\begin{enumerate}
				\item[\AM 1] $(g\circ f)(\x_1) = g(f(\x_1)) \subset g(\x_2 \sqcup \Z) \subset \x_3 \sqcup \Z$~;
				\item[\AM 2] $(g \circ f)(\ex_1) = g(f(\ex_1)) \subset g(\ex_2 \sqcup \Z) \subset \ex_3 \sqcup \Z$~;
				\item[\AM 3] Let $(x_1, \ldots, x_n)$ be a $((g \circ f),\Sigma_1,\Sigma_3)$-biadmissible sequence. We claim that $(x_1, \ldots, x_n)$ is $(f,\Sigma_1,\Sigma_2)$-biadmissible  and that $(f(x_1), \ldots, f(x_n))$ is $(g,\Sigma_2,\Sigma_3)$-biadmissible. Indeed, since $f$ satisfies \AM 2, an exchangeable variable $x \in \ex_1$ is sent either to an exchangeable variable in $\ex_2$ or to an integer. If $f(x) \in \Z$, because $g$ is a ring homomorphism, then $g(f(x)) \in \Z$ and therefore $(g \circ f)(x)$ is not exchangeable, a contradiction. Thus, $(x_1, \ldots, x_n)$ is $(f,\Sigma_1,\Sigma_2)$-biadmissible. It follows that $(f(x_1), \ldots, f(x_n))$ is $\Sigma_2$-admissible and therefore, as $(g(f(x_1)), \ldots, g(f(x_n)))$ is $\Sigma_3$-admissible, the sequence $(f(x_1), \ldots, f(x_n))$ is $(g,\Sigma_2,\Sigma_3)$-biadmissible. 

				Now, since $f$ satisfies \AM 3, for every $i$ such that $1 \leq i \leq n$ and any $x \in \x_1$, we have
				$$f(\mu_{x_i} \circ \cdots \circ \mu_{x_1}(x)) = \mu_{f(x_i)} \circ \cdots \circ \mu_{f(x_1)}(f(x)).$$
				and, since $g$ satisfies \AM 3, for any $(g,\Sigma_2,\Sigma_3)$-biadmissible sequence $(y_1, \ldots, y_n)$ and any $i$ such that $1 \leq i \leq n$ and $y \in \x_2$, we have
				$$g(\mu_{y_i} \circ \cdots \circ \mu_{y_1}(y)) = \mu_{g(y_i)} \circ \cdots \circ \mu_{g(y_1)}(g(y)).$$
				Then, as $f(\x_1) \subset \x_2 \sqcup \Z$, we get
				$$(g\circ f)(\mu_{x_i} \circ \cdots \circ \mu_{x_1}(x)) = \mu_{(g\circ f)(x_i)} \circ \cdots \circ \mu_{(g\circ f)(x_1)}((g\circ f)(x))$$
				for any $x \in \x$ so that $(g \circ f)$ satisfies \AM 3.
			\end{enumerate}
			Thus, $g \circ f : \mathcal A_1 \fl \mathcal A_3$ is a rooted cluster morphism.
		\end{proof}

		Therefore, we can set the following definition~:
		\begin{defi}
			The \emph{category of rooted cluster algebras} is the category $\amr$ defined by~:
			\begin{itemize}
				\item The objects in $\amr$ are the rooted cluster algebras~;
				\item The morphisms between two rooted cluster algebras are the rooted cluster morphisms.
			\end{itemize}
		\end{defi}

		\begin{rmq}
			One should observe the importance of the condition \AM 2 in the proof of Proposition \ref{prop:composition} in order to obtain well-defined compositions of rooted cluster morphisms. If this condition is removed from the definition of a rooted cluster morphism, one can easily construct examples of rooted cluster morphisms whose composition is not a rooted cluster morphism. 

			For instance, consider $\Sigma_1 = (x,x,[0])$, $\Sigma_2 = (x,\emptyset, [0])$ and $\Sigma_3 = \left((x,y),(x,y),\left[\begin{array}{rr} 0 & 1 \\ -1 & 0 \end{array}\right]\right)$. Let $f$ denote the identity in $\Q(x)$ and let $g$ denote the canonical inclusion of $\Q(x)$ in $\Q(x,y)$. By construction $f$ and $g$ satisfy \AM 1 and not \AM 2. Since there are neither non-empty $(f,\Sigma_1,\Sigma_2)$-biadmissible sequences nor non-empty $(g,\Sigma_2,\Sigma_3)$-biadmissible sequences, $f$ and $g$ satisfy \AM 3. The composition $g \circ f$ also satisfies \AM 1 and $(x)$ is $(g \circ f,\Sigma_1,\Sigma_3)$-biadmissible. However, 
			$$(g \circ f)(\mu_x(x)) = \frac{2}{x} \neq \frac{1+y}{x} = \mu_{(g \circ f)(x)}((g \circ f)(x)).$$
			Therefore, $g \circ f$ does not satisfy \AM 3.
		\end{rmq}

	\subsection{Ideal rooted cluster morphisms}
		\begin{defi}[Image seed]
			Let $f : \mathcal A(\Sigma) \fl \mathcal A(\Sigma')$ be a rooted cluster morphism. The \emph{image seed of $\Sigma$ under $f$} is
			$$f(\Sigma) = (\x' \cap f(\x), \ex' \cap f(\ex), B'[f(\x)]).$$
		\end{defi}
	
		\begin{lem}\label{lem:mutimage}
			Let $f : \mathcal A(\Sigma) \fl \mathcal A(\Sigma')$ be a rooted cluster morphism and let $(y_1, \ldots, y_l)$ be an $f(\Sigma)$-admissible sequence. Then $(y_1, \ldots, y_l)$ is $\Sigma'$-admissible and
			$$\mu_{y_l} \circ \cdots \mu_{y_1, f(\Sigma)}(y_1) = \mu_{y_l} \circ \cdots \mu_{y_1, \Sigma'}(y_1).$$
		\end{lem}
		\begin{proof}
			Let $\Sigma = (\x,\ex,B)$ and $\Sigma'=(\x',\ex',B')$. Because $\mu_{y_1}(B'[f(\x)]) = B'[\mu_{y_1}(f(\x))]$, it is enough to prove the statement for $l=1$ and to proceed by induction. By definition, exchangeable variables in $f(\Sigma)$ are exchangeable in $\Sigma'$. Now, if $y_1$ is $f(\Sigma)$-admissible, we have
			$$\mu_{y_1,f(\Sigma)}(y_1) = \frac{1}{y_1} \left( \prod_{\substack{b'_{zy_1}>0 \\ z \in \x' \cap f(\x)}}z^{b'_{zy_1}} + \prod_{\substack{b'_{zy_1}<0 \\ z \in \x' \cap f(\x)}}z^{-b'_{zy_1}}\right).$$
			Because $y_1$ is exchangeable in $f(\Sigma)$, there exists some $x_1 \in \ex$ such that $y_1=f(x_1)$ and we have
			$$\mu_{x_1,\Sigma}(x_1) = \frac{1}{x_1} \left( \prod_{\substack{b_{ux_1}>0 \\ u \in \x}}u^{b_{ux_1}} + \prod_{\substack{b_{ux_1}<0 \\ u \in \x}}u^{-b_{ux_1}}\right).$$
			Therefore, since $f$ satisfies \AM 3, we get
			$$\mu_{y_1,\Sigma'}(y_1) = f(\mu_{x_1,\Sigma}(x_1)) = \frac{1}{y_1} \left( \prod_{\substack{b_{ux_1}>0 \\ u \in \x}}f(u)^{b_{ux_1}} + \prod_{\substack{b_{ux_1}<0 \\ u \in \x}}f(u)^{-b_{ux_1}}\right).$$
			And by definition 
			$$\mu_{y_1,\Sigma'}(y_1) = \frac{1}{y_1} \left( \prod_{\substack{b'_{zy_1}>0 \\ z \in \x'}}z^{b'_{zy_1}} + \prod_{\substack{b'_{zy_1}<0 \\ z \in \x'}}z^{-b'_{zy_1}}\right).$$
			Therefore, $b'_{zy_1} = 0$ for any $z \in \x' \setminus f(\x)$ so that $\mu_{y_1,\Sigma'}(y_1) = \mu_{y_1,f(\Sigma)}(y_1)$.
		\end{proof}

		\begin{lem}\label{lem:image}
			Let $f : \mathcal A(\Sigma) \fl \mathcal A(\Sigma')$ be a rooted cluster morphism. Then $\mathcal A(f(\Sigma)) \subset f(\mathcal A(\Sigma))$.
		\end{lem}
		\begin{proof}
			Let $\Sigma = (\x,\ex,B)$ and $\Sigma'=(\x',\ex',B')$. We have to prove that any cluster variable in $\mathcal A(f(\Sigma))$ belongs to $f(\mathcal A(\Sigma))$. Fix an arbitrary cluster variable $y \in \mathcal A(f(\Sigma))$, then there exists an $f(\Sigma)$-admissible sequence $(y_1, \ldots, y_l)$ such that $y =  \mu_{y_l} \circ \cdots \circ \mu_{y_1}(y_1)$. 

			We claim that any $f(\Sigma)$-admissible sequence $(y_1, \ldots, y_l)$ lifts to an $(f,\Sigma,f(\Sigma))$-biadmissible sequence $(x_1, \ldots, x_l)$. If $l=1$, then $y_1 =f(x_1)$ for some $x_1 \in \ex$ and thus $(x_1)$ is biadmissible. Assume now that for $k <l$ the sequence $(y_1, \ldots, y_k)$ lifts to a biadmissible sequence $(x_1, \ldots, x_k)$. Then $y_{k+1}$ is exchangeable in $\mu_{y_k} \circ \cdots \circ \mu_{y_1}(f(\Sigma))$ and thus there exists some $x \in \ex$ such that $y_{k+1} = \mu_{y_k} \circ \cdots \circ \mu_{y_1}(f(x))$. Because $f$ satisfies $\AM3$, we get 
			$$y_{k+1} = f(\mu_{x_k} \circ \cdots \circ \mu_{x_1}(x)).$$
			Therefore, $x_{k+1} = \mu_{x_k} \circ \cdots \circ \mu_{x_1}(x)$ is exchangeable in $\mu_{x_k} \circ \cdots \circ \mu_{x_1}(\Sigma)$ and $(y_1, \ldots, y_{k+1})$ lifts to $(x_1, \ldots, x_{k+1})$. The claim follows by induction.

			If $l=0$, then by definition of $f(\Sigma)$, the elements in the cluster of $f(\Sigma)$ belong to $f(\mathcal A(\Sigma))$. If $l \geq 0$, then it follows from the claim that $(y_1, \ldots, y_l)$ lifts to an $(f, \Sigma, f(\Sigma))$-biadmissible sequence $(x_1, \ldots, x_l)$ in $\mathcal A(\Sigma)$. Moreover, we have 
			\begin{align*}
				y 
					& = \mu_{y_l} \circ \cdots \circ \mu_{y_1,f(\Sigma)}(y_1)\\
					& = \mu_{f(x_l)} \circ \cdots \circ \mu_{f(x_1),f(\Sigma)}(f(x_1))\\
					& = \mu_{f(x_l)} \circ \cdots \circ \mu_{f(x_1),\Sigma'}(f(x_1))\\
					& = f(\mu_{x_l} \circ \cdots \circ \mu_{x_1,\Sigma}(x_1)) \in f(\mathcal A(\Sigma)).
			\end{align*}
			where the third equality follows from Lemma \ref{lem:mutimage} and the last one from the fact that $f$ satisfies \AM 3.
		\end{proof}

		\begin{defi}[Ideal rooted cluster morphism]
			A rooted cluster morphism $f : \mathcal A(\Sigma) \fl \mathcal A(\Sigma')$ is called \emph{ideal} if $\mathcal A(f(\Sigma)) = f(\mathcal A(\Sigma))$.
		\end{defi}

		We shall meet along the article several natural classes of rooted cluster morphisms which are ideal (see for instance Corollary \ref{corol:injideal}, Proposition \ref{prop:res} or Proposition \ref{prop:specialisation}). However, we do not know whether or not every rooted cluster morphism is ideal. We may thus state the following problem~:
		\begin{prob}
			Characterise the rooted cluster morphisms which are ideal.
		\end{prob}

		\begin{exmp}
			Consider the seeds
			$$\Sigma = \left((x_1,x_2,x_3),(x_2),
				\left[\begin{array}{rrr}
					0 & 1 & 0 \\
					-1 & 0 & 1 \\
					0 & -1 & 0
				\end{array}\right]
			\right)
			\textrm{ and }
			\Sigma' = \left((y_1,y_2,y_3),(y_1,y_2,y_3),
				\left[\begin{array}{rrr}
					0 & 1 & 0 \\
					-1 & 0 & 1 \\
					0 & -1 & 0			                  
				\end{array}\right]
			\right)$$
			and the ring homomorphism 
			$$f: \left\{\begin{array}{rcl}
				\mathcal F_{\Sigma} & \fl & \mathcal F_{\Sigma'} \\
				x_1 & \mapsto & 1 \\
				x_2 & \mapsto & y_1 \\
				x_3 & \mapsto & y_2
			\end{array}\right.$$
			The only $(f,\Sigma,\Sigma')$-biadmissible sequence to consider is $(x_2)$ and 
			$$f(\mu_{x_2}(x_2)) = f\left(\frac{x_1+x_3}{x_2}\right) = \frac{1+y_2}{y_1} = \mu_{y_1}(y_1) = \mu_{f(x_2)}(f(x_2))$$
			so that $f$ is a rooted cluster morphism $\mathcal A(\Sigma) \fl \mathcal A(\Sigma')$. Moreover we have 
			$$f(\Sigma) = \left((y_1,y_2),(y_1)
				\left[\begin{array}{rr}
					0 & 1 \\
					-1 & 0
			\end{array}\right]\right).$$
			Therefore
			$$\mathcal A(\Sigma) = \Z[x_1,x_2,x_3,\frac{x_1+x_3}{x_2}] \text{ and }\mathcal A(f(\Sigma)) = \Z[y_1,y_2,\frac{1+y_2}{y_1}]$$
			so that $\mathcal A(f(\Sigma)) = f(\mathcal A(\Sigma))$ and therefore $f$ is ideal.
		\end{exmp}

		\begin{defi}[Rooted cluster ideal]
			A (ring theoretic) ideal $I$ in a rooted cluster algebra $\mathcal A(\Sigma)$ is called a \emph{rooted cluster ideal} if the quotient $\mathcal A(\Sigma)/I$ can be endowed with a structure of rooted cluster algebra such that the canonical projection is a rooted cluster morphism.
		\end{defi}

		\begin{prop}
			Let $f: \mathcal A(\Sigma) \fl \mathcal A(\Sigma')$ be an ideal rooted cluster morphism. Then $\ker(f)$ is a rooted cluster ideal.
		\end{prop}
		\begin{proof}
			Let $f: \mathcal A(\Sigma) \fl \mathcal A(\Sigma')$ be an ideal rooted cluster morphism. Then $f$ induces a ring isomorphism 
			$$\mathcal A(\Sigma) /\ker(f) \simeq f(\mathcal A(\Sigma)) = \mathcal A(f(\Sigma))$$
			endowing $\mathcal A(\Sigma) /\ker(f)$ with a structure of rooted cluster algebra with initial seed $f(\Sigma)$.

			Moreover, since $f$ is a rooted cluster morphism, the morphism $\bar f: \mathcal A(\Sigma) \fl \mathcal A(f(\Sigma))$ induced by $f$ is also a rooted cluster morphism and therefore $\mathcal A(\Sigma) \fl \mathcal A(\Sigma)/\ker(f)$ is a rooted cluster morphism. 
		\end{proof}

\section{Rooted cluster isomorphisms}\label{section:iso}
	In this section we characterise isomorphisms in $\amr$ which we call \emph{rooted cluster isomorphisms}. These results are generalisations of those obtained in \cite{ASS:automorphisms} for coefficient-free skew-symmetric cluster algebras. We recall that an \emph{isomorphism} in $\amr$ is an invertible morphism.

	We start with a general lemma on surjective morphisms~:
	\begin{lem}\label{lem:surj}
		Let $\Sigma=(\x_1,\ex_1,B^1)$ and $\Sigma_2=(\x_2,\ex_2,B^2)$ be two seeds and $f: \mathcal A(\Sigma_1) \fl \mathcal A(\Sigma_2)$ be a surjective ring homomorphism satisfying \AM 1. Then $\x_2 \subset f(\x_1)$ and $\ex_2 \subset f(\ex_1)$.
	\end{lem}
	\begin{proof}
		Let $z \in \x_2$. Since $f$ is surjective, there exists $y \in \mathcal A(\Sigma_1)$ such that $f(y)=z$. According to the Laurent phenomenon, there exists a Laurent polynomial $L$ such that $y=L(x|x \in \x_1)$. Therefore, $z = f(y) = L(f(x)|x \in \x_1)$. Since $f$ satisfies \AM 1, we know that $f(\x_1) \subset \x_2 \sqcup \Z$. If $z \not \in f(\x_1)$, since $\x_2$ is a transcendence basis of $\mathcal F_{\Sigma_2}$, we get a contradiction. Therefore, $z \in f(\x_1)$ and thus $\x_2 \subset f(\x_1)$.

		Fix now $z \in \ex_2$. According to the above discussion, we know that $f^{-1}(z) \cap \x_1 \neq \emptyset$. Since $f$ is surjective, there exists $X \in \mathcal A(\Sigma_1)$ such that $f(X) = \mu_{z,\Sigma_2}(z)$. Now we know that $X \in \Z[\x_1 \setminus \ex_1][\ex_1^{\pm 1}]$. Therefore, if $f^{-1}(z) \cap \ex_1 = \emptyset$, then $f^{-1}(z) \cap \x_1 \subset (\x_1 \setminus \ex_1)$ and therefore, $X$ is a sum of Laurent monomials with non-negative partial degree with respect to any $t \in f^{-1}(z) \cap \x_1$. Therefore, $f(X)$ is a sum of Laurent monomials with non-negative partial degree with respect to $z$, a contradiction since $\mu_{z,\Sigma_2}(z) = f(X)$ is the sum of two Laurent monomials with partial degree -1 with respect to $z$. Therefore $f^{-1}(z) \cap \ex_1 \neq \emptyset$ so that $\ex_2 \subset f(\ex_1)$.
	\end{proof}

	\begin{corol}\label{corol:bij}
		Let $\Sigma_i=(\x_i,\ex_i,B^i)$ be a seed for $i \in \ens{1,2}$ and let $f : \mathcal A(\Sigma_1) \fl \mathcal A(\Sigma_2)$ be a bijective ring homomorphism satisfying \AM 1. Then $f$ induces a bijection from $\x_1$ to $\x_2$. If moreover $f$ satisfies \AM 2, then $f$ induces a bijection from $\ex_1$ to $\ex_2$.
	\end{corol}
	\begin{proof}
		Since $f$ is injective and satisfies \AM 1, it induces an injection $\x_1 \fl \x_2$. Since $f$ is also surjective, it follows from Lemma \ref{lem:surj} that $\x_2 \subset f(\x_1)$ and $\ex_2 \subset f(\ex_2)$. Therefore, $f$ induces a bijection from $\x_1$ to $\x_2$. If moreover $f$ satisfies \AM 2, then $f$ induces an injection from $\ex_1$ to $\ex_2$ and thus it induces a bijection from $\ex_1$ to $\ex_2$.
	\end{proof}

	\begin{defi}[Isomorphic seeds]
		Two seeds $\Sigma=(\x_1,\ex_1,B^1)$ and $\Sigma_2=(\x_2,\ex_2,B^2)$ are called \emph{isomorphic} if there exists a bijection $\phi : \x_1 \fl \x_2$ inducing a bijection $\phi: \ex_1 \fl \ex_2$ and such that $b^2_{\phi(x),\phi(y)} = b^1_{xy}$ for every $x,y \in \x_1$. We then write $\Sigma_1 \simeq \Sigma_2$ and $B^1 \simeq B^2$.
	\end{defi}

	\begin{defi}[Opposite seed]
		Given a seed $\Sigma=(\x,\ex,B)$, the \emph{opposite seed} is $$\Sigma^{\op} = (\x,\ex,-B).$$
	\end{defi}

	\begin{defi}[Simplification of a seed]\label{defi:simplification}
		Given a seed $\Sigma=(\x,\ex,B)$, we set $\overline B=(\overline b_{xy})_{x,y \in \x} \in M_{\x}(\Z)$ where
		$$\overline b_{xy} = \left\{\begin{array}{ll}
			0 & \textrm{ if } x,y \in \x \setminus \ex \\
			b_{xy} & \textrm{ otherwise. }
		\end{array}\right.$$
		The \emph{simplification of the seed} $\Sigma$ is defined as $\overline \Sigma=(\x,\ex,\overline B)$.
	\end{defi}

	\begin{rmq}
		In terms of valued quivers, simplifying the seed simply corresponds to removing all the arrows between the frozen points. An example is shown in Figure \ref{fig:quivertriangulation} where the arrows between frozen points are shown dashed.
	\end{rmq}

	\begin{defi}[Locally rooted cluster morphism]
		Let $\Sigma_1 = (\x_1,\ex_1,B^1)$ and $\Sigma_2 = (\x_2,\ex_2,B^2)$ be two seeds. A ring homomorphism $f$ from $\mathcal A(\Sigma_1)$ to $\mathcal A(\Sigma_2)$ is called a \emph{locally rooted cluster morphism} if it satisfies \AM 1, \AM 2 and
		\begin{tabularx}{15cm}{rX}
			\AM {3loc}: & for any $x \in \ex_1$ and any $y \in \x_1$, we have, $f(\mu_{x,\Sigma_1}(y)) = \mu_{f(x),\Sigma_2}(f(y))$.
		\end{tabularx}
	\end{defi}

	As we now prove, for bijective ring homomorphisms, it is possible to simplify considerably the condition \AM 3 (compare \cite[Proposition 2.4]{ASS:automorphisms}).
	\begin{lem}\label{lem:amlocambij}
		Let $\Sigma_1$ and $\Sigma_2$ be two seeds and let $f: \mathcal A(\Sigma_1) \fl \mathcal A(\Sigma_2)$ be a bijective locally rooted cluster morphism. Then~:
		\begin{enumerate}
			\item $f$ is a rooted cluster morphism~;
			\item $\overline {\Sigma_1} \simeq \overline {\Sigma_2}$ or $\overline {\Sigma_1} \simeq (\overline{\Sigma_2})^{\op}$.
		\end{enumerate}
	\end{lem}
	\begin{proof}
		For $i \in \ens{1,2}$, we set $\Sigma_i=(\x_i,\ex_i,B^i)$. Let $f$ be as in the hypothesis. It follows from Corollary \ref{corol:bij} that $f$ induces bijections from $\x_1$ to $\x_2$ and from $\ex_1$ to $\ex_2$.

		For every $x \in \ex_1$, we have
		\begin{align*}
			f(\mu_x(x)) 
				& = f\left(\frac{1}{x} \left(\prod_{\substack{z \in \x_1~;\\ b^1_{xz}>0}} z^{b^1_{xz}} + \prod_{\substack{z \in \x_1~;\\ b^1_{xz}<0}} z^{-b^1_{xz}}\right)\right) \\
				& = \frac{1}{f(x)} \left(\prod_{\substack{z \in \x_1~;\\ b^1_{xz}>0}} f(z)^{b^1_{xz}} + \prod_{\substack{z \in \x_1~;\\ b^1_{xz}<0}} f(z)^{-b^1_{xz}}\right). \\
		\end{align*}
		Because $f$ satisfies \AM {3loc}, we have
		\begin{align*}
			f(\mu_x(x)) 
				& = \mu_{f(x)}(f(x)) \\
				& = \frac{1}{f(x)} \left(\prod_{\substack{f(z) \in \x_2~;\\ b^2_{f(x)f(z)}>0}} f(z)^{b^2_{f(x)f(z)}} + \prod_{\substack{f(z) \in \x_2~;\\ b^2_{f(x)f(z)}<0}} f(z)^{-b^2_{f(x)f(z)}}\right)
		\end{align*}
		and because $\x_2$ is algebraically independent, we get $b^1_{xz} = b^2_{f(x)f(z)}$ for any $x \in \ex_1$ and any $z \in \x_1$ or $b^1_{xz} = -b^2_{f(x)f(z)}$ for any $x \in \ex_1$ and any $z \in \x_1$, that is, $\overline {B^1} \simeq \overline {B^2}$ or $\overline {B^1} \simeq (-\overline {B^2})$. Therefore, $\overline {\Sigma_1} \simeq \overline {\Sigma_2}$ or $\overline {\Sigma_1} \simeq (\overline \Sigma_2)^{\op}$.

		Since the mutations in $\mathcal A(\Sigma_1)$ and $\mathcal A(\Sigma_2)$ are entirely encoded in the simplifications $\overline {\Sigma_1}$ and $\overline {\Sigma_2}$ of the exchange matrices, it follows easily that $f$ is a rooted cluster morphism.
	\end{proof}

	\begin{theorem}\label{theorem:iso}
		Let $\mathcal A(\Sigma_1)$ and $\mathcal A(\Sigma_2)$ be two rooted cluster algebras. Then $\mathcal A(\Sigma_1)$ and $\mathcal A(\Sigma_2)$ are isomorphic in $\amr$ if and only if $\overline \Sigma_1 \simeq \overline \Sigma_2$ or $\overline \Sigma_1 \simeq \overline \Sigma_2^{\op}$.
	\end{theorem}
	\begin{proof}
		As $\am$ is a concrete category, an isomorphism $f:\mathcal A(\Sigma_1) \fl \mathcal A(\Sigma_2)$ is necessarily bijective. Therefore, Corollary \ref{corol:bij} and Lemma \ref{lem:amlocambij} imply that $f$ induces a bijection $\x_1 \xrightarrow{\sim} \x_2$ such that $\overline {B^1} \simeq \overline {B^2}$ or $\overline {B^1} \simeq (\overline {B^2})^{\op}$. Moreover, it also follows from Corollary \ref{corol:bij} that $f$ is a bijection $\ex_1 \xrightarrow{\sim} \ex_2$ so that it induces an isomorphism of seeds $\overline \Sigma_1 \simeq \overline \Sigma_2$ or $\overline \Sigma_1 \simeq \overline \Sigma_2^{\op}$.

		Conversely, if $\overline \Sigma_1 \simeq \overline \Sigma_2$ or $\overline \Sigma_1 \simeq \overline \Sigma_2^{\op}$, we consider the bijection $\sigma : \x_1 \fl \x_2$ inducing the isomorphism of seeds. It thus induces naturally a ring isomorphism $f_\sigma : \mathcal F_{\Sigma_1} \fl \mathcal F_{\Sigma_2}$ and it is easily seen that $f_\sigma$ is a rooted cluster isomorphism.
	\end{proof}

	\begin{corol}\label{corol:iso}
		The isomorphisms in $\amr$ coincide with the bijective rooted cluster morphisms.
	\end{corol}
	\begin{proof}
		As $\amr$ is a concrete category, isomorphisms are bijective. Conversely, if we consider a bijective rooted cluster morphism $\mathcal A(\Sigma_1) \fl \mathcal A(\Sigma_2)$, then it follows from Lemma \ref{lem:amlocambij} that $f$ induces an isomorphism of seeds $\overline \Sigma_1 \simeq \overline \Sigma_2$ or $\overline \Sigma_1 \simeq \overline \Sigma_2^{\op}$ so that $f$ is an isomorphism in $\amr$.
	\end{proof}

	\begin{corol}\label{corol:isoclus}
		Let $\mathcal A(\Sigma_1)$ and $\mathcal A(\Sigma_2)$ be two rooted cluster algebras and let $f: \mathcal A(\Sigma_1) \fl \mathcal A(\Sigma_2)$ be a rooted cluster isomorphism. Then the following hold~:
		\begin{enumerate}
			\item any $\Sigma_1$-admissible sequence is $(f,\Sigma_1,\Sigma_2)$-biadmissible~;
			\item any $\Sigma_2$-admissible sequence lifts to a unique $(f,\Sigma_1,\Sigma_2)$-biadmissible sequence~;
			\item $f(\mathcal X_{\Sigma_1}) = \mathcal X_{\Sigma_2}$. \hfill \qed
		\end{enumerate}
	\end{corol}

	\begin{rmq}
		For non-bijective morphisms, one can find locally rooted cluster morphisms which are not rooted cluster morphisms. For instance, let $\Sigma_1=(\x_1,\ex_1,B^1)$ and $\Sigma_2=(\x_2,\ex_2,B^2)$ 
		where
		$$\ex_1 = \x_1= \ens{x_1,x_2,x_3}, \quad \ex_2 = \x_2 = \ens{u_1,u_2},$$
		$$B^1= \left[\begin{array}{rrr} 0 & 1 & 0 \\ -1 & 0 & -1 \\ 0 & 1 & 0 \end{array}\right]
		\textrm{ and }
		B^2= \left[\begin{array}{rr} 0 & 1 \\ -2 & 0 \end{array}\right]$$
		and consider the ring homomorphism
		$$\pi : \left\{\begin{array}{rcl}
				\Q(x_1,x_2,x_3) & \fl & \Q(u_1,u_2) \\
				x_1,x_3 & \mapsto & u_1 \\
				x_2 & \mapsto & u_2. \\
		\end{array}\right.$$

		Then
		$$\pi(\mu_{x_1}(x_1)) = \pi(\mu_{x_3}(x_3)) = \pi \left(\frac{1+x_2}{x_1}\right) = \frac{1+u_2}{u_1} = \mu_{u_1}(u_1)$$
		and
		$$\pi(\mu_{x_2}(x_2)) = \pi \left(\frac{1+x_1x_3}{x_2}\right) = \frac{1+u_1^2}{u_2} = \mu_{u_2}(u_2)$$
		so that $\pi$ commutes with biadmissible sequences of length one. 

		However
		$$\mu_{x_2}(\x_1) = \ens{x_1,x_2'=\frac{1+x_1x_3}{x_2},x_3}$$
		$$(\mu_{x_1} \circ \mu_{x_2})(\x_1) = \ens{\frac{1+x_2+x_1x_3}{x_1x_2},x_2',x_3}$$
		and
		$$(\mu_{x_2'} \circ \mu_{x_1} \circ \mu_{x_2})(\x_1) = \ens{\frac{1+x_2+x_1x_3}{x_1x_2},\frac{1+x_2}{x_1},x_3}$$
		but
		$$\mu_{u_2}(\x_2) = \ens{u_1,u_2'=\frac{1+u_1^2}{u_2}}$$
		$$(\mu_{u_1} \circ \mu_{u_2})(\x_2) = \ens{\frac{1+u_2+u_1^2}{u_1u_2},u_2'}$$
		and
		$$(\mu_{u_2'} \circ \mu_{u_1} \circ \mu_{u_2})(\x_2) = \ens{\frac{1+u_2+u_1^2}{u_1u_2},\frac{1+2u_2+u_2^2+u_1^2}{u_1^2u_2}}.$$
		so that
		$$\pi\left(\frac{1+x_2}{x_1}\right) = \frac{1+u_2}{u_1} \neq \frac{1+2u_2+u_2^2+u_1^2}{u_1^2u_2}$$
		and thus $\pi$ is not a rooted cluster morphism between the cluster algebras $\mathcal A(\Sigma_1)$ and $\mathcal A(\Sigma_2)$.

		Note that the morphism $\pi$ is induced by the folding of the quiver $\xymatrix{Q_{B^1}: 1 & \ar[l] 2 \ar[r] & 3}$ with respect to the automorphism group exchanging 1 and 3. For general results concerning the interactions of foldings with cluster algebras, we refer the reader for instance to \cite{Demonet:categorification}.
	\end{rmq}

	\begin{rmq}
		\begin{enumerate}
			\item[a)] Two rooted cluster algebras associated with mutation-equivalent seeds are \emph{not} necessarily isomorphic in the category $\amr$ since mutation-equivalent seeds are in general neither isomorphic nor opposite.
			\item[b)] The \emph{cluster automorphisms} considered in \cite{ASS:automorphisms} correspond in our context to rooted cluster isomorphisms from $\mathcal A(\Sigma)$ to itself when $\Sigma$ is finite, skew-symmetric and coefficient-free. The groups of cluster automorphisms have been computed for seeds associated with Dynkin or affine quivers, see \cite[\S 3.3]{ASS:automorphisms}.
			\item[c)] A strong isomorphism $\mathcal A(\Sigma_1) \fl \mathcal A(\Sigma_2)$ in the sense of \cite{cluster2} is a rooted cluster isomorphism such that $\overline{\Sigma_1} \simeq \overline{\Sigma_2}$.
		\end{enumerate}
	\end{rmq}

\section{Rooted cluster monomorphisms}\label{section:mono}
	We recall that a \emph{monomorphism} in a category is a morphism $f$ such that if there exist morphisms $g$ and $h$ such that $fg=fh$, then $g=h$.

	\begin{lem}\label{lem:injfrozen}
		Let $\Sigma =(\x,\ex,B)$ be a seed, let $\y \subset \x$ and let $\Theta = (\y,\emptyset,C)$ be another seed. Then the canonical ring homomorphism $\mathcal F_{\Theta} \fl \mathcal F_{\Sigma}$ induces an injective rooted cluster morphism $\mathcal A(\Theta) \fl \mathcal A(\Sigma)$.
	\end{lem}
	\begin{proof}
		The canonical ring homomorphism $\mathcal F_{\Theta} \fl \mathcal F_{\Sigma}$ sends $x$ to $x$ for any $x \in \y$ therefore, it satisfies \AM 1 and \AM 2. Moreover, since there are no exchangeable variables in $\Theta$, it automatically satisfies \AM 3. We thus only have to prove that it induces a ring homomorphism $\mathcal A(\Theta) \fl \mathcal A(\Sigma)$, and this is clear because $\mathcal A(\Theta) = \Z[x\ | \ x \in \y] \subset \Z[x \ | \ x \in \x] \subset \mathcal A(\Sigma)$.
	\end{proof}

	\begin{rmq}
		In order for Lemma \ref{lem:injfrozen} to hold, it is necessary to consider non-invertible coefficients because if the image of a frozen variable is exchangeable then the image of its inverse would have to be the inverse of the exchangeable variable, which is not in the cluster algebra.
	\end{rmq}

	\begin{prop}\label{prop:mono}
		Monomorphisms in $\amr$ coincide with injective rooted cluster morphisms.
	\end{prop}
	\begin{proof}
		Let $\Sigma_i = (\x_i,\ex_i,B^i)$ be seeds for $1 \leq i \leq 3$ and consider rooted cluster morphisms
		$$\xymatrix{\mathcal A(\Sigma_1) \ar@<+2pt>[r]^g \ar@<-2pt>[r]_h & \mathcal A(\Sigma_2) \ar[r]^f & \mathcal A(\Sigma_3)}.$$
		Since $\amr$ is a concrete category, every injective morphism in $\amr$ is a monomorphism. We thus only need to prove the converse.

		Let $f$ be a non-injective rooted cluster morphism. Because it satisfies \AM 1, we get $f(\x_2) \subset \x_3 \sqcup \Z$. If $f(\x_2) \subset \x_3$ and if the restriction of $f$ to $\x_2$ is injective, $f$ sends a transcendence basis of $\mathcal F_{\Sigma_2}$ to an algebraically independent family in $\mathcal F_{\Sigma_3}$ and therefore, it induces an injective ring homomorphism $\mathcal F_{\Sigma_2} \fl \mathcal F_{\Sigma_3}$ so that it is itself injective, a contradiction. Thus, there are two cases to consider~:
		\begin{itemize}
			\item there exists $x \in \x_2$ such that $f(x) \in \Z$,
			\item there exist $x, y \in \x_2$ such that $x \neq y$ and $f(x) = f(y)$.
		\end{itemize}
		
		In the first case, it follows from Lemma \ref{lem:injfrozen} that we can consider the rooted cluster morphisms $h,g : \Z[x] \fl \mathcal A(\Sigma_2)$ given by $g(x)=x$ and $h(x)=f(x) \in \Z$. Then $fh(x) = fg(x) = f(x)$ so that $fh = fg$ but $f \neq g$. Thus $f$ is not a monomorphism in $\amr$.

		In the second case, it also follows from Lemma \ref{lem:injfrozen} that we can consider the rooted cluster morphisms $h,g: \Z[x,y] \fl \mathcal A(\Sigma_2)$ given by $g(x) = h(y) = x$ and  $g(y) = h(x) = y$. Thus, $fg=fh$ but $g \neq h$ and therefore $f$ is not a monomorphism in $\amr$.
	\end{proof}

	As a consequence, the study of monomorphisms in $\amr$ restricts to the study of injective rooted cluster morphisms.
	\begin{lem}\label{lem:injsubseed}
		Let $\Sigma_1 = (\x_1,\ex_1,B_1)$ and $\Sigma_2=(\x_2,\ex_2,B^2)$ be two seeds and let $f: \mathcal A(\Sigma_1) \fl \mathcal A(\Sigma_2)$ be an injective rooted cluster morphism. Then $f$ induces an isomorphism of seeds $\overline \Sigma_1 \simeq \overline{f(\Sigma_1)}$ or $\overline \Sigma_1 \simeq (\overline{f(\Sigma_1)})^{\op}$. 
	\end{lem}
	\begin{proof}
		As $f$ is injective and satisfies \AM 1, we have $f(\x_1) \subset \x_2$ and since it satisfies \AM 2, we have $f(\ex_1) = \ex_2 \cap f(\x_1)$. Let $x \in \ex_1$. Since $f$ satisfies \AM 3, we have $f(\mu_{x,\Sigma_1}(x)) = \mu_{f(x),\Sigma_2}(f(x))$ so that 
		$$\prod_{\substack{y \in \x_1,\\ b^1_{xy}>0}} f(y)^{b^1_{xy}} + \prod_{\substack{y \in \x_1,\\ b^1_{xy}<0}} f(y)^{-b^1_{xy}} = \prod_{\substack{z \in \x_2,\\ b^2_{f(x)z}>0}} z^{b^2_{f(x)z}} + \prod_{\substack{z \in \x_2,\\ b^2_{f(x)z}<0}} z^{-b^2_{f(x)z}}.$$
		Therefore, $b^2_{f(x)z} = 0$ for any $z \not \in f(\x_1)$ and either $b^2_{f(x)z}=b^1_{xy}$ for any $y \in \x_1$ such that $f(y)=z$ or $b^2_{f(x)z}=-b^1_{xy}$ for any $y \in \x_1$ such that $f(y)=z$, which proves the lemma.
	\end{proof}

	\begin{corol}\label{corol:injideal}
		Any injective rooted cluster morphism is ideal.
	\end{corol}
	\begin{proof}
		Let $\Sigma = (\x,\ex,B)$ and $\Sigma'=(\x',\ex',B')$ be two seeds and let $f: \mathcal A(\Sigma) \fl \mathcal A(\Sigma')$ be an injective rooted cluster morphism. As a cluster algebra does only depend on the simplification of the seed, we can assume that both $\Sigma = \overline \Sigma$ and $\Sigma' = \overline {\Sigma'}$. According to Lemma \ref{lem:injsubseed}, $f$ induces an isomorphism of seeds $\Sigma \simeq f(\Sigma)$ or $\Sigma \simeq \overline{f(\Sigma)}$. It follows that every $\Sigma$-admissible sequence is $(f,\Sigma,\Sigma')$-biadmissible. 

		In order to prove that $f$ is ideal, it is enough to prove that for any cluster variable $x$ in $\mathcal A(\Sigma)$, the variable $f(x)$ is an element in $\mathcal A(f(\Sigma))$. Let thus $z$ be a cluster variable in $\mathcal A(\Sigma)$. Then either $z$ is a frozen variable in $\Sigma$ and thus $f(z)$ is a frozen variable in $f(\Sigma)$ and we are done, or there exists a $\Sigma$-admissible sequence $(x_1, \ldots, x_l)$ such that $z = \mu_{x_l} \circ \cdots \circ \mu_{x_1}(x)$ for some $x \in \ex$. Then $(x_1, \ldots, x_l)$ is $(f,\Sigma,\Sigma')$-biadmissible and because $f$ satisfies \AM 3, we get
		$$f(z) = f(\mu_{x_l} \circ \cdots \circ \mu_{x_1,\Sigma}(x)) = \mu_{f(x_l)} \circ \cdots \circ \mu_{f(x_1),\Sigma'}(f(x))$$
		but it follows from Lemma \ref{lem:mutimage} that 
		$$\mu_{f(x_l)} \circ \cdots \circ \mu_{f(x_1),\Sigma'}(f(x)) = \mu_{f(x_l)} \circ \cdots \circ \mu_{f(x_1),f(\Sigma)}(f(x)).$$
		Therefore, $f(z)$ is a cluster variable in $\mathcal A(f(\Sigma))$ and thus $f(\mathcal A(\Sigma)) \subset \mathcal A(f(\Sigma))$. The reverse inclusion follows from Lemma \ref{lem:image} and therefore $f$ is ideal.
	\end{proof}

	As a byproduct of the proof of Corollary \ref{corol:injideal}, we obtain~:
	\begin{corol}\label{corol:clustervarinj}
		Let $\Sigma_1,\Sigma_2$ be two seeds and $f: \mathcal A(\Sigma_1) \fl \mathcal A(\Sigma_2)$ be an injective rooted cluster morphism. Then~:
		\begin{enumerate}
			\item any $\Sigma_1$-admissible sequence is $(f,\Sigma_1,\Sigma_2)$-biadmissible,
			\item $f(\mathcal X_{\Sigma_1}) \subset \mathcal X_{\Sigma_2}$. \hfill \qed
		\end{enumerate}
	\end{corol}

	We recall that a seed $\Sigma=(\x,\ex,B)$ is called~:
	\begin{itemize}
		\item of \emph{finite cluster type} if $\mathcal X_\Sigma$ is finite,
		\item \emph{acyclic} if the valued quiver obtained by deleting the arrows between frozen vertices in $Q_{B}$ has no oriented cycles,
		\item \emph{mutation-finite} if the mutation class of the exchange matrix of $B$ is finite.
	\end{itemize}

	\begin{corol}
		Let $\Sigma_1,\Sigma_2$ be two seeds and $f: \mathcal A(\Sigma_1) \fl \mathcal A(\Sigma_2)$ be an injective rooted cluster morphism. Then~:
		\begin{enumerate}
			\item If $\Sigma_2$ is of finite cluster type, then so is $\Sigma_1$~;
			\item If $\Sigma_2$ is acyclic, then so is $\Sigma_1$~;
			\item If $\Sigma_2$ is mutation-finite, then so is $\Sigma_1$.
		\end{enumerate}
	\end{corol}
	\begin{proof}
		The first assertion is a consequence of Corollary \ref{corol:clustervarinj}. The second assertion follows from Lemma \ref{lem:injsubseed} and from the fact that a subquiver of an acyclic (valued) quiver is acyclic. The third assertion follows from Lemma \ref{lem:injsubseed} and from the fact that a full subquiver of a mutation-finite (valued) quiver is mutation-finite.
	\end{proof}

	\subsection{Monomorphisms arising from triangulations of the $n$-gon}\label{ssection:typeA}
		For any integer $m \geq 3$, we denote by $\Pi_m$ the $m$-gon whose points are labelled cyclically from 1 to $m$. For $m \geq 4$, the cluster algebra $\mathcal A(\Pi_m)$ (with coefficients associated with boundary arcs) is a cluster algebra of type $A_{m-3}$. 

		We construct by induction a family $\ens{T_m}_{m \geq 3}$ where each $T_m$ is a \emph{fan triangulation} of $\Pi_m$. We start with the triangle $\Pi_3$ whose points are denoted by 1,2 and 3. For any $m \geq 3$, the triangulation $T_{m+1}$ of $\Pi_{m+1}$ is obtained by gluing a triangle along the boundary arc joining $m$ to 1 in $T_m$ and the new marked point introduced by this triangle is labelled by $m+1$, as shown in the figure below.

		
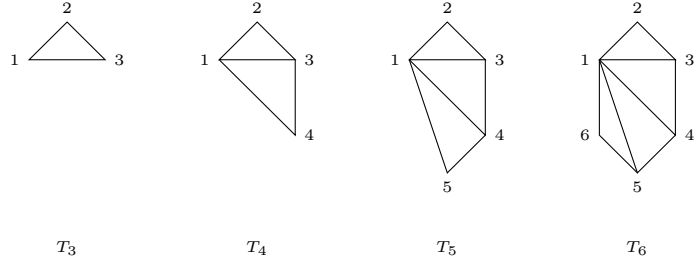
\begin{center}
			\begin{tikzpicture}[scale = .5]
				\tikzstyle{every node}=[font=\tiny]

				\foreach \x in {0,5,10,15}
				{
					\draw (\x+0,2) -- (\x+1,3) -- (\x+2,2) -- cycle;
					\fill (\x+0,2) node [left] {1};
					\fill (\x+1,3) node [above] {2};
					\fill (\x+2,2) node [right] {3};
				}

				\foreach \x in {5,10,15}
				{
					\draw (\x+2,2) -- (\x+2,0) -- (\x,2);
					\fill (\x+2,0) node [right] {4};
				}

				\foreach \x in {10,15}
				{
					\draw (\x+2,0) -- (\x+1,-1) -- (\x,2);
					\fill (\x+1,-1) node [below] {5};
				}

				\foreach \x in {15}
				{
					\draw (\x+1,-1) -- (\x+0,0) -- (\x,2);
					\fill (\x+0,0) node [left] {6};
				}

				\fill (1,-3) node {$T_3$};
				\fill (6,-3) node {$T_4$};
				\fill (11,-3) node {$T_5$};
				\fill (16,-3) node {$T_6$};
			\end{tikzpicture}
		\end{center}

		For any $m \geq 4$, we denote by $\Sigma_m$ the seed associated with $T_m$ in $\mathcal A(\Pi_m)$. Cluster variables in $\mathcal A(\Pi_m)$ are identified with the arcs joining two marked points in $\Pi_m$ and for any $i$ and $j$ such that $1 \leq i<j \leq m$, we denote by $x_{ij}$ the variable corresponding to the arc joining $i$ to $j$. The exchangeable variables are thus the variables corresponding to internal arcs. The exchange relations given by the mutations in $\mathcal A(\Pi_m)$ are the so-called \emph{Plücker relations}~:
		$$x_{ij}x_{kl} = x_{ik}x_{jl} + x_{il}x_{kj} \text{ for } 1 \leq i < j < k < l \leq m.$$

		For any $m' \geq m$, the inclusion of $T_m$ in $T_{m'}$ defines a natural ring monomorphism $j_{m,m'}:\mathcal F_{\Sigma_m} \fl \mathcal F_{\Sigma_{m'}}$. Since arcs (or internal arcs) in $T_m$ are sent to arcs (or internal arcs, respectively) in $T_{m'}$, then $j_{m,m'}$ satisfies \AM 1 and \AM 2. Moreover, since exchange relations in $\mathcal A(\Pi_m)$ and $\mathcal A(\Pi_{m'})$ correspond to Plücker relations in $\Pi_m$ and $\Pi_{m'}$ respectively, it is easily seen that $j_{m,m'}$ satisfies \AM 3. Finally, since every admissible sequence of variables in $\mathcal A(\Pi_m)$ is $(j_{m,m'},\Sigma_m,\Sigma_{m'})$-biadmissible, it follows from the fact that $j_{m,m'}$ preserves the Plücker relations that it commutes with biadmissible mutations and that $j_{m,m'}(\mathcal A(\Pi_m)) \subset \mathcal A(\Pi_{m'})$.

		Therefore, for $m < m'$, we have exhibited an injective rooted cluster morphism
		$$j_{m,m'} : \mathcal A(\Pi_m) \fl \mathcal A(\Pi_{m'}).$$

		\begin{exmp}
			Consider the cluster algebra $\mathcal A(\Pi_4)$ whose cluster variables are shown on the square below.
			
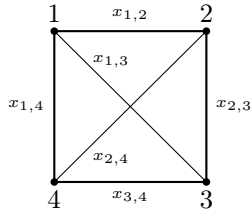
\begin{center}
				\begin{tikzpicture}
			% 		\tikzstyle{every node}=[font=\tiny]
					\draw[thick] (0,0) -- (0,2) -- (2,2) -- (2,0) -- cycle;
					\fill (0,0) circle (.05);
					\fill (2,0) circle (.05);
					\fill (0,2) circle (.05);
					\fill (2,2) circle (.05);

					\draw (0,0) -- (2,2);
					\draw (0,2) -- (2,0);

					\fill (0,2) node [above] {1};
					\fill (2,2) node [above] {2};
					\fill (2,0) node [below] {3};
					\fill (0,0) node [below] {4};

					\fill (1,2) node [above] {\tiny $x_{1,2}$};
					\fill (1,0) node [below] {\tiny $x_{3,4}$};
					\fill (0,1) node [left] {\tiny $x_{1,4}$};
					\fill (2,1) node [right] {\tiny $x_{2,3}$};

					\fill (.75,1.6) node {\tiny $x_{1,3}$};
					\fill (.75,.3) node {\tiny $x_{2,4}$};
				\end{tikzpicture}
			\end{center}
			Then 
			$$\mathcal A(\Pi_4) = \Z[x_{ij} \ | \ 1 \leq i < j \leq 4]/(x_{13}x_{24} = x_{12}x_{34} + x_{14}x_{23}).$$

			Now consider the cluster algebra $\mathcal A(\Pi_5)$. It has 4 additional cluster variables, two are exchangeable and two are frozen. We show these new variables on the picture below.
			
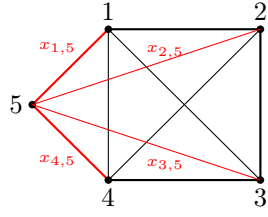
\begin{center}
				\begin{tikzpicture}
			
					\draw[thick,red] (0,0) -- (-1,1);
					\draw[thick,red] (0,2) -- (-1,1);
					\draw[thick] (0,2) -- (2,2) -- (2,0) --(0,0);
					\draw (0,0) -- (0,2);

					\fill (-1,1) circle (.05);
					\fill (0,0) circle (.05);
					\fill (2,0) circle (.05);
					\fill (0,2) circle (.05);
					\fill (2,2) circle (.05);

					\draw (0,0) -- (2,2);
					\draw (0,2) -- (2,0);

					\fill (-1,1) node [left] {5};
					\fill (0,2) node [above] {1};
					\fill (2,2) node [above] {2};
					\fill (2,0) node [below] {3};
					\fill (0,0) node [below] {4};

					\draw[red] (-1,1) -- (2,2);
					\draw[red] (-1,1) -- (2,0);

					\fill[red] (.75,1.7) node {\tiny $x_{2,5}$};
					\fill[red] (.75,.2) node {\tiny $x_{3,5}$};

					\fill[red] (-.3,1.75) node [left] {\tiny $x_{1,5}$};
					\fill[red] (-.3,.25) node [left] {\tiny $x_{4,5}$};
				\end{tikzpicture}
			\end{center}
			Then 
			$$\mathcal A(\Pi_5) = \Z[x_{ij} \ | \ 1 \leq i < j \leq 5]/(x_{ij}x_{kl} = x_{ik}x_{jl} + x_{il}x_{kj} \text{ for } 1 \leq i < j < k < l \leq 5).$$

			In particular, the canonical morphism $j_{4,5}: \Q(x_{ij} \ | \ 1 \leq i < j \leq 4) \fl \Q(x_{ij} \ | \ 1 \leq i < j \leq 5)$ induces an injective ring homomorphism $\mathcal A(\Pi_4) \fl \mathcal A(\Pi_5)$ which is a rooted cluster monomorphism $\mathcal A(\Sigma_{T_4}) \fl \mathcal A(\Sigma_{T_5})$.
		\end{exmp}

		\begin{defi}[Full subseed of a seed]\label{defi:subseed}
			Given a seed $\Sigma=(\x,\ex,B)$ where $\x=(x_i,i \in I)$ and given a subset $J \subset I$, we set $\Sigma_{|J}$ the seed with cluster $\x_{|J}=(x_i, i \in J)$, with exchangeable variables $\ex_{|J} = \ex \cap \x_{|J}$ and exchange matrix $B[J]$. Such a seed is called a \emph{full subseed} of $\Sigma$.
		\end{defi}

		\begin{rmq}
			If $\Sigma'$ is a full subseed of $\Sigma$, the canonical morphism $\mathcal F_{\Sigma'} \fl \mathcal F_{\Sigma}$ does not necessarily induce an injective rooted cluster morphism $\mathcal A(\Sigma') \fl \mathcal A(\Sigma)$. For instance, if $\Sigma_m$ denotes the coefficient-free seed associated with the quiver 
			$$Q_m : 1 \fl 2 \fl \cdots \fl m$$
			for any $m \geq 1$, then the canonical inclusion $\iota : \mathcal F_{\Sigma_m} \fl \mathcal F_{\Sigma_{m+1}}$ does not induce a rooted cluster morphism $\mathcal A(\Sigma_m) \fl \mathcal A(\Sigma_{m+1})$ because 
			$$\iota\left(\mu_{x_{m},\Sigma_m}(x_m)\right) = \iota\left(\frac{1+x_{m-1}}{x_m}\right) \neq \frac{x_{m-1}+x_{m+1}}{x_m} = \mu_{x_m,\Sigma_{m+1}}(x_m).$$

			However, as we shall see in Proposition \ref{prop:res}, if $\Sigma'$ is a full subseed of $\Sigma$, then there is a natural ideal surjective rooted cluster morphism $\mathcal A(\Sigma) \fl \mathcal A(\Sigma')$.
		\end{rmq}

	We now describe a combinatorial operation on seeds which allows one to construct a class of injective rooted cluster morphisms.

	\subsection{Amalgamated sum of seeds}\label{ssection:amalgam}
		Let $\Sigma_1=(\x_1,\ex_1,B^1)$ and $\Sigma_2=(\x_2,\ex_2,B^2)$ be two seeds and $\mathcal A(\Sigma_1)$, $\mathcal A(\Sigma_2)$ be the corresponding rooted cluster algebras. 

		Let $\Delta^1 \subset (\x_1 \setminus \ex_1)$ and $\Delta^2 \subset (\x_2 \setminus \ex_2)$ be two (possibly empty) subsets such that there is an isomorphism of seeds $\Sigma_{1|\Delta_1} \simeq \Sigma_{2|\Delta_2}$. In this case, we say that $\Sigma_1$ and $\Sigma_2$ are \emph{glueable along $\Delta_1$ and $\Delta_2$}.

		Let $\Delta$ be a family of undeterminates in bijection with $\Delta_1$ and $\Delta_2$. We set 
		$$\x_1 \coprod_{\Delta_1,\Delta_2} \x_2 = (\x_1 \setminus \Delta_1) \sqcup (\x_2 \setminus \Delta_2) \sqcup \Delta.$$
		As $\Delta_1$ and $\Delta_2$ consist of frozen variables, $\ex_1$ and $\ex_2$ are naturally identified with two disjoint subsets of $\x_1 \coprod_{\Delta_1,\Delta_2} \x_2$ and we set
		$$\ex_1 \coprod_{\Delta_1,\Delta_2} \ex_2 = \ex_1 \sqcup \ex_2.$$
		
		With respect to the partitions $\x_i = (\x_i \setminus \Delta_i) \sqcup \Delta$, for any $i \in \ens{1,2}$ the matrix $B^i$ can be written~:
		$$B^i = \left[\begin{array}{r|r}
			B^i_{11} & B^i_{12} \\
			\hline
			B^i_{21} & B_{\Delta}
		\end{array}\right]$$
		where the matrices are possibly infinite.

		We then set
		$$B^1 \coprod_{\Delta_1,\Delta_2} B^2 = \left[\begin{array}{c|c|c}
			B^1_{11} & 0 & B^1_{12}\\
			\hline
			0 & B^2_{11} & B^2_{12}\\
			\hline
			B^1_{21} & B^2_{12} & B_{\Delta}
		\end{array}\right]$$

		\begin{defi}[Amalgamated sum of seeds]
			With the above notations, the \emph{amalgamated sum of $\mathcal A(\Sigma_1)$ and $\mathcal A(\Sigma_2)$ along $\Delta_1,\Delta_2$} is the rooted cluster algebra $\mathcal A(\Sigma)$ where $\Sigma=(\x,\ex,B)$ with~:
			\begin{enumerate}
				\item $\x = \x_1 \coprod_{\Delta_1,\Delta_2} \x_2$~;
				\item $\ex = \ex_1 \coprod_{\Delta_1,\Delta_2} \ex_2$~;
				\item $B= B^1 \coprod_{\Delta_1,\Delta_2} B^2$.
			\end{enumerate}
			
			We use the notations $$\Sigma = \Sigma_1 \coprod_{\Delta_1, \Delta_2} \Sigma_2 \text{
			and 
			}\mathcal A(\Sigma) = \mathcal A(\Sigma_1) \coprod_{\Delta_1, \Delta_2} \mathcal A(\Sigma_2).$$
		\end{defi}

		\begin{rmq}
			In terms of valued quivers, the amalgamated sum of exchange matrices corresponds to the amalgamated sum of valued quivers. For instance, the following figure shows an example of amalgamated sum over the subquivers in the shaded area where points corresponding to exchangeable (or frozen) variables are black (or white, respectively).
			
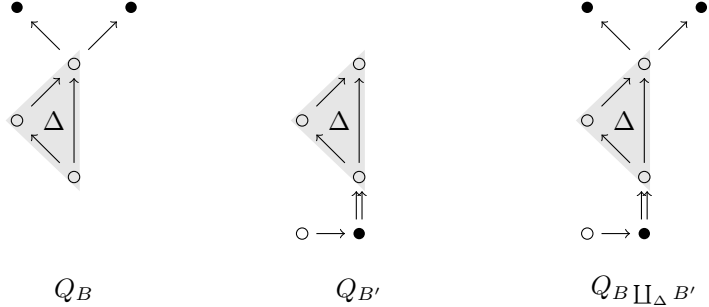
\begin{center}
				\begin{tikzpicture}[scale = .75]
				
					\foreach \x in {0,5,10}
					{
						\fill[gray!20] (\x-.2,0) -- (\x+1.1,1.25) -- (\x+1.1,-1.25) -- cycle;

						\fill (\x+.65,0) node {$\Delta$};

						\draw[->] (\x+.25,.25) -- (\x+.75,.75);
						\draw[<-] (\x+.25,-.25) -- (\x+.75,-.75);
						\draw[<-] (\x+1,.75) -- (\x+1,-.75);

						\draw (\x+0,0) circle (.1);
						\draw (\x+1,1) circle (.1);
						\draw (\x+1,-1) circle (.1);
					}

					\draw[->] (.75,1.25) -- (.25,1.75);
					\draw[->] (1.25,1.25) -- (1.75,1.75);
					\fill (0,2) circle (.1);
					\fill (2,2) circle (.1);

					\fill (1,-3) node {$Q_{B}$};

					\draw[->] (5.95,-1.75) -- (5.95,-1.25);
					\draw[->] (6.05,-1.75) -- (6.05,-1.25);
					\fill (6,-2) circle (.1);
					\draw[->] (5.25,-2) -- (5.75,-2);
					\draw (5,-2) circle (.1);

					\fill (6,-3) node {$Q_{B'}$};

					\draw[->] (10.75,1.25) -- (10.25,1.75);
					\draw[->] (11.25,1.25) -- (11.75,1.75);
					\fill (10,2) circle (.1);
					\fill (12,2) circle (.1);

					\draw[->] (10.95,-1.75) -- (10.95,-1.25);
					\draw[->] (11.05,-1.75) -- (11.05,-1.25);
					\fill (11,-2) circle (.1);
					\draw[->] (10.25,-2) -- (10.75,-2);
					\draw (10,-2) circle (.1);

					\fill (11,-3) node {$Q_{B \coprod_{\Delta} B'}$};
				\end{tikzpicture}
			\end{center}
		\end{rmq}

		\begin{lem}\label{lem:injcoprod}
			Let $\Sigma_1$ and $\Sigma_2$ be two seeds which are glueable along subsets $\Delta_1$ and $\Delta_2$ as above and let $\Sigma = \Sigma_1 \coprod_{\Delta_1, \Delta_2} \Sigma_2$. Then for any $i$ such that $i \in \ens{1,2}$, the morphism $\mathcal F_{\Sigma_i} \fl \mathcal F_{\Sigma}$ induced by the inclusion induces an injective rooted cluster morphism $\mathcal A(\Sigma_i) \fl \mathcal A(\Sigma)$.
		\end{lem}
		\begin{proof}
			By construction of $\Sigma$, the canonical morphism $j_i: \mathcal F_{\Sigma_i} \fl \mathcal F_{\Sigma}$ is injective and satisfies \AM 1 and \AM 2. We now prove by induction on $l$ that any $\Sigma_1$-admissible sequence of length $l$ is $(j_1,\Sigma_1,\Sigma)$-biadmissible and that $j_1$ commutes with mutations along biadmissible sequences of length $l$.

			Let $x \in \ex_1$. Then $j_1(x) = x \in \ex$ so that $(x)$ is $(j_1,\Sigma_1,\Sigma)$-biadmissible and therefore mutating in $\Sigma_1$ gives
			$$\mu_x(x) = \prod_{\substack{y \in \x_1~;\\ b^1_{xy}>0}} y^{b^1_{xy}} + \prod_{\substack{y \in \x_1~;\\ b^1_{xy}<0}} y^{-b^1_{xy}}$$
			and, since $b_{zy} =0$ for any $z \in (\x_1 \setminus \Delta) \supset \ex_1$ and any $y \in \x_2 \setminus \Delta$, mutating in $\Sigma$ gives
			\begin{align*}
				\mu_{x,\Sigma}(x) 
					& = \prod_{\substack{y \in \x~;\\ b_{xy}>0}} y^{b_{xy}} + \prod_{\substack{y \in \x~;\\ b_{xy}<0}} y^{-b_{xy}}\\
					& = \prod_{\substack{y \in (\x_1 \setminus \Delta_1) \sqcup \Delta~;\\ b_{xy}>0}} y^{b_{xy}} + \prod_{\substack{y \in (\x_1 \setminus \Delta_1) \sqcup \Delta~;\\ b_{xy}<0}} y^{-b_{xy}}\\
					& = \prod_{\substack{y \in \x_1~;\\ b^1_{xy}>0}} y^{b^1_{xy}} + \prod_{\substack{y \in \x_1~;\\ b^1_{xy}<0}} y^{-b^1_{xy}}\\
					& = j_1(\mu_{x,\Sigma_1}(x))
			\end{align*}
			and thus $j_1$ is a locally rooted cluster morphism.

			Assume now that we proved the claim for any $k<l$ and let $\Sigma^{(k)} = \mu_{x_k} \circ \cdots \circ \mu_{x_1}(\Sigma)$ and $\Sigma_1^{(k)} = \mu_{x_k} \circ \cdots \circ \mu_{x_1}(\Sigma_1)$. We denote by $B^{(k)}=(b^{(k)}_{xy})$ (or $B^{1,(k})=(b^{1,(k)}_{xy})$) the exchange matrix of $\Sigma^{(k)}$ (or $\Sigma_1^{(k)}$, respectively). By the induction hypothesis, the cluster $\x^{(k)}$ of $\Sigma^{(k)}$ is $\x_1^{(k)} \sqcup (\x_2  \setminus \Delta) \sqcup \Delta$ where $\x_1^{(k)} \sqcup \Delta_1$ is the cluster of $\Sigma_1^{(k)}$. Then an easy induction proves that
			$$b^{(l)}_{xy} = \left\{\begin{array}{ll}
				0 & \text{ if } x \in \x_1^{(k)} \text{ and } y \in \x_2 \setminus \Delta \\
				b^{1,(l)}_{xy} & \text{ if } x,z \in \x_1^{(k)} \sqcup \Delta.
			\end{array}\right.$$
			Because $j_1$ commutes with sequences of biadmissible mutations of length $k$, any variable $x_{k+1}$ exchangeable in $\Sigma_1^{(k)}$ is also exchangeable in $\Sigma^{(k)}$ and a similar calculation proves that for any such exchangeable variable $x_{k+1}$, the morphism $j_1$ commutes with $\mu_{x_{k+1}} \circ \cdots \circ \mu_{x_1}$.
		\end{proof}

		\begin{rmq}
			If $\Delta_1$ and $\Delta_2$ do not consist of frozen variables, then the canonical morphism $\mathcal F_{\Sigma_i} \fl \mathcal F_{\Sigma}$ may not satisfy \AM 3. Indeed, if one considers the seeds
			$$\Sigma_1 = \left((x_1,x_2),(x_1,x_2),\left[\begin{array}{rr} 0 & 1 \\ -1 & 0 \end{array}\right]\right)$$
			$$\Sigma_2 = \left((x_2,x_3),(x_3),\left[\begin{array}{rr} 0 & 1 \\ -1 & 0 \end{array}\right]\right)$$
			so that
			$$\Sigma = \Sigma_1 \coprod_{x_2,x_2} \Sigma_2 = \left((x_1,x_3,x_2),(x_1,x_3,x_2),\left[\begin{array}{rrr} 0 & 0 & 1 \\ 0 & 0 & -1 \\ -1 & 1 & 0 \end{array}\right]\right).$$
			Then mutating along $x_2$ in $\Sigma_1$ gives $\frac{1+x_1}{x_2}$ whereas mutating along $x_2$ in $\Sigma$ gives $\frac{x_1+x_3}{x_2}$. Therefore, the canonical morphism $\mathcal F_{\Sigma_1} \fl \mathcal F_{\Sigma}$ does not satisfy \AM 3.
		\end{rmq}

		If $\Sigma_1=(\x_1,\ex_1,B^1)$ and $\Sigma_2=(\x_2,\ex_2,B^2)$ are glueable along $\Delta_1,\Delta_2$ and if we denote by $\Delta$ the common image of $\Delta_1$ and $\Delta_2$ in $\Sigma = \Sigma_1 \coprod_{\Delta_1,\Delta_2} \Sigma_2 = (\x,\ex,B)$, then the compositions of the canonical ring homomorphisms
		$$\Z[\Delta] \simeq \Z[\Delta_1] \fl \mathcal A(\Sigma_1) \text{ and }\Z[\Delta] \simeq \Z[\Delta_2] \fl \mathcal A(\Sigma_2)$$
		induce $\Z[\Delta]$-algebra structures on $\mathcal A(\Sigma_1)$ and $\mathcal A(\Sigma_2)$. Also the inclusion $\Z[\Delta] \subset \mathcal A(\Sigma)$ induces a $\Z[\Delta]$-algebra structure on $\mathcal A(\Sigma)$. Then we have the following proposition~:
		\begin{prop}
			The canonical ring isomorphism $\Q(\x) \xrightarrow{\sim} \Q(\x_1) \otimes_{\Q(\Delta)} \Q(\x_2)$ induces an isomorphism of $\Z[\Delta]$-algebras~:
			$$\mathcal A\left(\Sigma_1 \coprod_{\Delta_1,\Delta_2} \Sigma_2\right) \simeq \mathcal A(\Sigma_1) \otimes_{\Z[\Delta]} \mathcal A(\Sigma_2).$$	\hfill \qed
		\end{prop}

\section{Coproducts, products and amalgamated sums}\label{section:sums}
	\subsection{Coproducts of rooted cluster algebras}
		\begin{lem}\label{lem:sommes}
			The category $\amr$ admits countable coproducts. 
		\end{lem}
		\begin{proof}
			Let $I$ be a countable set and let $\ens{\mathcal A(\Sigma_i)}_{i \in I}$ be a countable family of rooted cluster algebras. For any $i \in I$, we set $\Sigma^i=(\x_i,\ex_i,B^i)$ and 
			$$\Sigma = \coprod_{i \in I} \Sigma_i = (\x,\ex,B)$$
			where
			$$\x = \bigsqcup_{i \in I} \x_i, \quad \ex = \bigsqcup_{i \in I} \ex_i,$$
			and $B$ is the block-diagonal matrix whose blocks are indexed by $I$ and such that for any $i \in I$, the $i$-th diagonal block is $B^i$. Then $B$ is locally finite and thus $\Sigma$ is a well-defined seed. For any $i \in I$, we denote by $j_i$ the canonical inclusion $\mathcal F_{\Sigma_i} \fl \mathcal F_\Sigma$, which clearly induces a rooted cluster morphism $\mathcal A(\Sigma_i) \fl \mathcal A(\Sigma)$.

			Now assume that there exists a rooted cluster algebra $\mathcal A(\Theta)$ and for any $i \in I$ a rooted cluster morphism $g_i:\mathcal A(\Sigma_i) \fl \mathcal A(\Theta)$. In order to prove that $\mathcal A(\Sigma)$ is the coproduct of the $\mathcal A(\Sigma_i)$ with $i \in I$, we need to prove that there exists a unique rooted cluster morphism $h:\mathcal A(\Sigma) \fl \mathcal A(\Theta)$ such that $h \circ j_i = g_i$ for any $i \in I$. It is easily seen that such a rooted cluster morphism $h$ exists if and only if there exists a ring homomorphism $h: \mathcal F_\Sigma \fl \mathcal F_\Theta$ satisfying $h(x)=g_i(x)$ for any $i \in I$. This latter condition defines precisely one ring homomorphism $h:\mathcal F_\Sigma \fl \mathcal F_\Theta$ and therefore, $h$ exists and is unique. 
		\end{proof}

		\begin{corol}
			The full subcategory of $\amr$ formed by rooted cluster algebras associated with finite seeds has finite coproducts.
		\end{corol}

		\begin{rmq}		 
			Topologically, this can be interpreted by saying that if $(S,M)$ is a marked surface, with connected components $S_1, \ldots, S_n$ and with $M_i = M \cap S_i$ for any $1\leq i \leq n$, then $\mathcal A(S,M)$ is the coproduct of the $\mathcal A(S_i,M_i)$ in $\amr$. Indeed, for any $1 \leq i \leq n$, let $T_i$ be a triangulation of $(S_i,M_i)$. Then $T = \bigsqcup_i T_i$ is a triangulation of $(S,M)$ and it follows immediately from the definitions that $\Sigma_T = \coprod_{i=1}^n \Sigma_{T_i}$ so that $\mathcal A(\Sigma_T) = \coprod_{i=1}^n \mathcal A(\Sigma_{T_i})$.
		\end{rmq}

	\subsection{Products of rooted cluster algebras}
		\begin{prop}\label{prop:products}
			$\amr$ does not generally admit products.
		\end{prop}
		\begin{proof}
			The proof consists of the construction of an example of two rooted cluster algebras whose product is not defined in $\amr$. We consider the rooted cluster algebras associated with the seeds $\Sigma_1 = ((t_1),\emptyset, [0])$ and $\Sigma_1 = ((t_2),\emptyset, [0])$, so that $\mathcal A(\Sigma_1) = \Z[t_1]$ and $\mathcal A(\Sigma_2) = \Z[t_2]$. Assume that there exists a product in $\amr$
			$$\xymatrix{
				& \mathcal A(\Sigma) \ar[ld]_{p_1} \ar[rd]^{p_2} \\
				\Z[t_1] && \Z[t_2].
			}$$
			with $\Sigma = (\x,\ex,B)$. For any $i \in \ens{1,2}$, the morphism $p_i$ is a rooted cluster morphism so that 
			$$p_i(\x) \subset \ens{t_i} \sqcup \Z \text{ and } p_i(\ex) \subset \Z.$$
			
			Consider the seed $\Sigma'=((x),\emptyset,[0])$ and for any $i \in \ens{1,2}$, let $f_i : \Z[x] \fl \Z[t_i]$ be the ring homomorphism sending $x$ to $t_i$. Then $f_i$ is a rooted cluster morphism $\mathcal A(\Sigma') =\Z[x] \fl \mathcal A(\Sigma)$ for any $i \in \ens{1,2}$. 
	
			By definition of the product, there exists a unique rooted cluster morphism $h: \mathcal A(\Sigma') \fl \mathcal A(\Sigma)$ such that the following diagram commutes:
			$$\xymatrix{
				& \mathcal A(\Sigma') \ar[d]^h \ar[ldd]_{f_1} \ar[rdd]^{f_2} \\
				& \mathcal A(\Sigma) \ar[ld]^{p_1} \ar[rd]_{p_2} \\
				\Z[t_1] && \Z[t_2].
			}$$
			In particular, for any $i \in \ens{1,2}$, there exists $x_i \in \x$ such that $p_i(x_i) = t_i$.

			Now consider the seed $\Sigma'= ((v_1,v_2), \emptyset, [0])$ so that $\mathcal A(\Sigma') = \Z[v_1,v_2]$. For any $i \in \ens{1,2}$, let $f_i:\mathcal A(\Sigma') \fl \mathcal A(\Sigma_i)$ be defined by $f_i(v_j) = \delta_{ij} t_j$ where $\delta_{ij}$ is the Kronecker symbol. Then each $f_i$ is a rooted cluster morphism. Again by definition of product, there exists a unique rooted cluster morphism $h: \mathcal A(\Sigma') \fl \mathcal A(\Sigma)$ such that the above diagram commutes. Since $h$ satisfies \AM 1, for any $i \in \ens{1,2}$, we have $h(v_i) = x_i$ for some $x_i \in \x$ such that $p_i(x_i) = p_i(x_i') = t_i$. If $x_i,x_i' \in \x$ are such that $p_i(x_i)=t_i$, then the morphism given by $h(v_i)=h(x_i')$ induces a rooted cluster morphism $\mathcal A(\Sigma') \fl \mathcal A(\Sigma)$ and by uniqueness, $h=h'$ and thus $x_i = x_i'$. Also, as $(p_1 \circ h)(v_1) = t_1$ and $(p_1 \circ h)(v_2) = 0$, we have $h(v_1) \neq h(v_2)$. Therefore, we obtained exactly two elements $h(v_1) = x_1$ and $h(v_2) = x_2$ in $\x$ such that $p_1(x_1) = t_1$, $p_2(x_1) = 0$, $p_1(x_2)=0$ and $p_2(x_2)=t_2$ and these elements are distinct.

			Now consider again the seed $\Sigma' = ((x),\emptyset,[0])$ and $f_i : \Z[x] \fl \Z[t_i]$ the morphism sending $x$ to $t_i$ for any $i \in \ens{1,2}$. Let $h$ be the unique morphism such that the above diagram commutes. Then, by commutativity of the left triangle, we get $h(x) = x_1$ and by commutativity of the right triangle, we get $h(x) = x_2$, a contradiction. Therefore, $\mathcal A(\Sigma_1)$ and $\mathcal A(\Sigma_2)$ have no product in $\amr$.
		\end{proof}

	\subsection{Amalgamated sums}
		In this subsection, we prove that the amalgamated sums of seeds yield pushouts of injective morphisms in $\amr$.

		Let $\Sigma_1=(\x_1,\ex_1,B^1)$ and $\Sigma_2=(\x_2,\ex_2,B^2)$ be two seeds and let $\Delta_1$, $\Delta_2$ be (possibly empty) subsets of $\x_1$ and $\x_2$ respectively such that $\Sigma_1$ and $\Sigma_2$ are glueable along $\Delta_1,\Delta_2$. We recall that necessarily $\Delta_1$ and $\Delta_2$ consist of frozen variables. 

		For any $k=1,2$, it follows from Lemma \ref{lem:injfrozen} that we have a natural injective rooted cluster morphism~:
		$$i_k : \Z[\Delta] \fl \mathcal A(\Sigma_k)$$
		and from Lemma \ref{lem:injcoprod} that we have a natural injective rooted cluster morphism
		$$j_k : \mathcal A(\Sigma_k) \fl \mathcal A(\Sigma_1) \coprod_{\Delta_1, \Delta_2} \mathcal A(\Sigma_2).$$

		\begin{prop}\label{prop:pushout}
			With the previous notations, the diagram
			$$\xymatrix{
				\displaystyle \Z[\Delta] \ar[rr]^{i_1} \ar[d]_{i_2} && \mathcal A(\Sigma_1) \ar[d]^{j_1}\\
				\mathcal A(\Sigma_2) \ar[rr]_{j_2} && \displaystyle \mathcal A(\Sigma_1) \coprod_{\Delta_1, \Delta_2} \mathcal A(\Sigma_2)
			}$$ 
			is the amalgamated sum of $i_1$ and $i_2$ in $\amr$.
		\end{prop}
		\begin{proof}
			One must first prove that the diagram commutes. Let $x$ be a variable in the cluster $\Delta$ of $\Sigma_1 \bigcap_{\Delta_1,\Delta_2} \Sigma_2$. Then $i_1$ identifies canonically $x$ with a variable $x_1$ in $\Delta_1$ and $i_2$ identifies canonically $x$ with a variable $x_2$ in $\Delta_2$. But $j_1$ and $j_2$ then identify canonically $x_1$ and $x_2$ with the variable $x$ viewed as an element in $\x_1 \coprod_{\Delta_1,\Delta_2} \x_2$. Thus the diagram commutes.

			Let now $\Sigma$ be a seed such that there exists a commutative diagram in $\amr$
			$$\xymatrix{
				\displaystyle \mathcal A(\Sigma_1 \bigcap_{\Delta_1,\Delta_2} \Sigma_2) \ar[rr]^{i_1} \ar[d]_{i_2} && \mathcal A(\Sigma_1) \ar[d]^{f_1}\\
				\mathcal A(\Sigma_2) \ar[rr]_{f_2} && \displaystyle \mathcal A(\Sigma).
			}$$
			We show that there exists a unique rooted cluster morphism $h : \mathcal A(\Sigma_1) \coprod_{\Delta_1, \Delta_2} \mathcal A(\Sigma_2) \fl \mathcal A(\Sigma)$ such that the following diagram commutes~:
			\begin{equation}\label{eq:diagrampushout}
				\xymatrix{
					& \displaystyle \mathcal A(\Sigma_1) \coprod_{\Delta_1, \Delta_2} \mathcal A(\Sigma_2) \ar[d]^h \\
					\mathcal A(\Sigma_1) \ar[ru]^{j_1} \ar[r]_{f_1} & \mathcal A(\Sigma) & \mathcal A(\Sigma_2). \ar[lu]_{j_2} \ar[l]^{f_2}
				}
			\end{equation}
			
			For any $k=1,2$ we identify $\Delta_k$ with $\Delta \subset \x_1 \coprod_{\Delta_1,\Delta_2} \x_2$ via the morphism $j_k$. It follows from the commutativity of the first diagram that $f_1(x)=f_2(x)$ for any $x \in \Delta$. We thus set $h : \mathcal F_{\Sigma_1 \coprod_{\Delta_1,\Delta_2} \Sigma_2} \fl \mathcal F_{\Sigma}$ via 
			$$h(x) = \left\{\begin{array}{ll}
				f_1(x) = f_2(x) & \textrm{ if } x \in \Delta~; \\
				f_1(x) & \textrm{ if } x \in \x_1~; \\
				f_2(x) & \textrm{ if } x \in \x_2.
			\end{array}\right.$$
			
			Because $f_1$ and $f_2$ are rooted cluster morphisms, $h$ is necessarily also a rooted cluster morphism and therefore the diagram \eqref{eq:diagrampushout} commutes. Conversely, if $h$ is a rooted cluster morphism such that the diagram \eqref{eq:diagrampushout} commutes, then $h$ is entirely determined by its values on $\x_1 \coprod_{\Delta_1,\Delta_2} \x_2$ and it is easily seen that $h$ must be the above morphism.
		\end{proof}

	\subsection{Topological interpretation of the amalgamated sums}\label{ssection:topsum}
		In this subsection we prove that amalgamated sums of cluster algebras of surfaces correspond to cluster algebras associated with connected sums of surfaces. 

		Let $(S_1,M_1)$ and $(S_2,M_2)$ be marked surfaces in the sense of \cite{FST:surfaces} and let $\d_1$ and $\d_2$ be boundary components respectively of $S_1$ and $S_2$ such that there exists a homeomorphism $h: \d_1 \xrightarrow{\sim} \d_2$ satisfying $h(\d_1 \cap M_1) = \d_2 \cap M_2$. 

		We denote by
		$$(S,M) = (S_1,M_1) \coprod_{\d_1,\d_2} (S_2,M_2)$$
		the \emph{connected sum} of $S_1$ and $S_2$ along $h$, that is, $S$ is the surface obtained by gluing $S_1$ and $S_2$ along the homeomorphism $h$ and $M = (M_1 \setminus (M_1 \cap \d_1)) \cup (M_2 \setminus (M_2 \cap \d_2)) \cup M_\d$ where $M_\d$ is the common image of $\d_1 \cap M_1$ and $\d_2 \cap M_2$ in the surface $S$ (see for instance \cite[p.8]{Massey:introduction}). We denote by $\d$ the common image of $\d_1$ and $\d_2$ in the surface $S$.

		\begin{rmq}
			Even if $(S_1,M_1)$ and $(S_2,M_2)$ are unpunctured surfaces, the surface $(S,M)$ may have punctures, see for instance Figure \ref{fig:collagebord}.
		\end{rmq}

		Let $T_1$ and $T_2$ be triangulations of $(S_1,M_1)$ and $(S_2,M_2)$ respectively and let $\Sigma_1$ and $\Sigma_2$ be the corresponding seeds. Then $\d_1$ is identified with a subset of the frozen variables in $\Sigma_1$ and $\d_2$ is identified with a subset of the frozen variables in $\Sigma_2$. In the connected sum $(S,M)$, the collection $T_1 \cup T_2$ defines a triangulation and we denote by $\Sigma$ the associated seed. Then it follows from Proposition \ref{prop:pushout} that
		$$\mathcal A(\Sigma^{\d}) = \mathcal A(\Sigma_1) \coprod_{\d_1, \d_2} \mathcal A(\Sigma_2)$$
		where the notation $\Sigma^{\d}$ means that we have frozen the variables corresponding to $\d$ in $\Sigma$
		
		\begin{corol}
			In the category $\amr$, the rooted cluster algebra associated with the triangulation $T$ of $S$ where the arcs in $\d$ are frozen is the amalgamated sum over the polynomial ring $\Z[\d]$ of the rooted cluster algebras $\mathcal A(\Sigma_1)$ and $\mathcal A(\Sigma_2)$ associated with the triangulations $T_1$ and $T_2$ induced by $T$ respectively on $S_1$ and $S_2$.
		\end{corol}

		\begin{exmp}
			Figure \ref{fig:collagebord} shows such a gluing.

			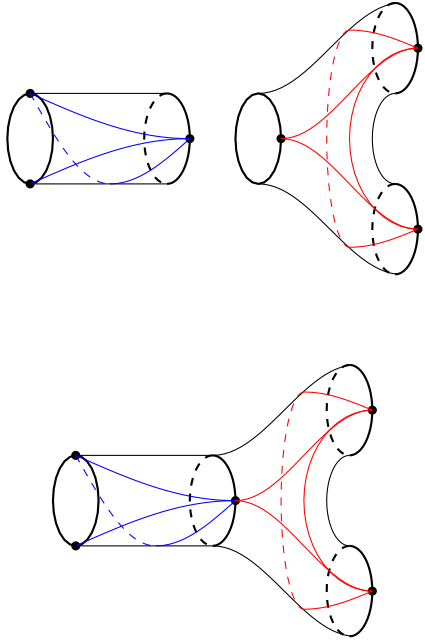
\begin{figure}[htb]
				\begin{center}
					\begin{tikzpicture}[scale = .6]

				% % % % % % % % % % % % % % % % % 
				% % % % % % % % Premier cylindre
				% % % % % % % % % % % % % % % % % 

						\draw[thick] plot[domain=0:2*pi] ({.5*cos(\x r)},{sin(\x r)});

						\draw[thick] plot[domain=pi/2:-pi/2] ({3+.5*cos(\x r)},{sin(\x r)});
						\draw[thick,dashed] plot[domain=pi/2:3*pi/2] ({3+.5*cos(\x r)},{sin(\x r)});
						
						\draw (0,1) -- (3,1);
						\draw (0,-1) -- (3,-1);

						% Marquage des points sur le cylindre
						\fill (0,1) circle (.1);
						\fill (0,-1) circle (.1);
						\fill (3.5,0) circle (.1);

						% Triangulation du cylindre 
						\draw[blue] plot[domain=0:1] ({(3.5)*\x},{1-sin((pi/2)*\x r)});
						\draw[blue] plot[domain=0:1] ({(3.5)*\x},{-1+sin((pi/2)*\x r)});
						\draw[blue,dashed] plot[domain=0:1] ({3.5*\x/2},{1-2*sin(pi/2*\x r)});
						\draw[blue] plot[domain=1:2] ({3.5*\x/2},{-sin(pi/2*\x r)});

				% % % % % % % % % % % % % % % % % % % % % % 
				% % % % % % % Seconde surface ; un pantalon
				% % % % % % % % % % % % % % % % % % % % % %

						\draw[thick] plot[domain=0:2*pi] ({5+.5*cos(\x r)},{sin(\x r)});

						\draw[thick] plot[domain=-pi/2:pi/2] ({8+.5*cos(\x r)},{2+sin(\x r)});
						\draw[thick,dashed] plot[domain=pi/2:3/2*pi] ({8+.5*cos(\x r)},{2+sin(\x r)});

						\draw[thick] plot[domain=-pi/2:pi/2] ({8+.5*cos(\x r)},{-2+sin(\x r)});
						\draw[thick,dashed] plot[domain=pi/2:3/2*pi] ({8+.5*cos(\x r)},{-2+sin(\x r)});
						
						\draw plot[domain=-pi/2:pi/2-.1] ({5+pi/2+\x},{2+sin(\x r)});
						\draw plot[domain=-pi/2:pi/2-.1] ({5+pi/2+\x},{-2-sin(\x r)});
						\draw plot[domain=pi/2:3*pi/2] ({8+.5*cos(\x r)},{sin(\x r)});

						% Marquage des points sur le pantalon
						\fill (8.5,2) circle (.1);
						\fill (8.5,-2) circle (.1);
						\fill (5.5,0) circle (.1);	

						% Une triangulation du pantalon
						\draw[red] plot[domain=pi/2:3*pi/2] ({8.5+1.5*cos(\x r)},{2*sin(\x r)});

						\draw[red] plot[domain=pi/2:pi] ({5.5+1.5*2/(pi)*\x},{2+.4*sin(\x r)});
						\draw[red] plot[domain=pi/2:pi] ({5.5+1.5*2/(pi)*\x},{-2-.4*sin(\x r)});
						\draw[red,dashed] plot[domain=pi/2:3*pi/2] ({7+.5*cos(\x r)},{2.4*sin(\x r)});

						\draw[red] plot[domain=pi/2:3*pi/2] ({4+3/pi*\x},{-1+sin(\x r)});
						\draw[red] plot[domain=pi/2:3*pi/2] ({4+3/pi*\x},{1-sin(\x r)});

				% % % % % % % % % % % % % % % % % % % % % % 
				% % % % % % % Collage des deux surfaces
				% % % % % % % % % % % % % % % % % % % % % % 

						% Le cylindre
						\draw[thick] plot[domain=0:2*pi] ({1+.5*cos(\x r)},{-8+sin(\x r)});

						\draw[thick] plot[domain=pi/2:-pi/2] ({4+.5*cos(\x r)},{-8+sin(\x r)});
						\draw[thick,dashed] plot[domain=pi/2:3*pi/2] ({4+.5*cos(\x r)},{-8+sin(\x r)});
						
						\draw (1,-7) -- (4,-7);
						\draw (1,-9) -- (4,-9);

						% Marquage des points sur le cylindre
						\fill (1,-7) circle (.1);
						\fill (1,-9) circle (.1);
						\fill (4.5,-8) circle (.1);

						% Le pantalon
						\draw[thick] plot[domain=-pi/2:pi/2] ({7+.5*cos(\x r)},{-6+sin(\x r)});
						\draw[thick,dashed] plot[domain=pi/2:3/2*pi] ({7+.5*cos(\x r)},{-6+sin(\x r)});

						\draw[thick] plot[domain=-pi/2:pi/2] ({7+.5*cos(\x r)},{-10+sin(\x r)});
						\draw[thick,dashed] plot[domain=pi/2:3/2*pi] ({7+.5*cos(\x r)},{-10+sin(\x r)});
						
						\draw plot[domain=-pi/2:pi/2-.1] ({4+pi/2+\x},{-6+sin(\x r)});
						\draw plot[domain=-pi/2:pi/2-.1] ({4+pi/2+\x},{-10-sin(\x r)});
						\draw plot[domain=pi/2:3*pi/2] ({7+.5*cos(\x r)},{-8+sin(\x r)});

						% Marquage des points sur le pantalon
						\fill (7.5,-6) circle (.1);
						\fill (7.5,-10) circle (.1);

						% Triangulation induite sur le cylindre
						\draw[blue] plot[domain=0:1] ({1+(3.5)*\x},{-8+1-sin((pi/2)*\x r)});
						\draw[blue] plot[domain=0:1] ({1+(3.5)*\x},{-8+-1+sin((pi/2)*\x r)});
						\draw[blue,dashed] plot[domain=0:1] ({1+3.5*\x/2},{-8+1-2*sin(pi/2*\x r)});
						\draw[blue] plot[domain=1:2] ({1+3.5*\x/2},{-8+-sin(pi/2*\x r)});

						% Triangulation induite sur le pantalon
						\draw[red] plot[domain=pi/2:3*pi/2] ({-1+8.5+1.5*cos(\x r)},{-8+2*sin(\x r)});

						\draw[red] plot[domain=pi/2:pi] ({-1+5.5+1.5*2/(pi)*\x},{-8+2+.4*sin(\x r)});
						\draw[red] plot[domain=pi/2:pi] ({-1+5.5+1.5*2/(pi)*\x},{-8+-2-.4*sin(\x r)});
						\draw[red,dashed] plot[domain=pi/2:3*pi/2] ({-1+7+.5*cos(\x r)},{-8+2.4*sin(\x r)});

						\draw[red] plot[domain=pi/2:3*pi/2] ({-1+4+3/pi*\x},{-8+-1+sin(\x r)});
						\draw[red] plot[domain=pi/2:3*pi/2] ({-1+4+3/pi*\x},{-8+1-sin(\x r)});

					\end{tikzpicture}
				\end{center}
				\caption{An example of connected sum of an annulus and a pair of pants along a boundary and the induced triangulations.}\label{fig:collagebord}
			\end{figure} 
		\end{exmp}

\section{Surjective rooted cluster morphisms}\label{section:epi}
	We recall that an \emph{epimorphism} in a category is a morphism $f$ such that if there exist morphisms $h$ and $g$ such that $gf = hf$, then $g = h$. In this section, we focus on surjective rooted cluster morphisms, which are particular cases of epimorphisms since $\amr$ is a concrete category.

	\begin{rmq}\label{rmq:cexepi}
		As in the category $\Ann$, epimorphisms in $\amr$ are \emph{not} necessarily surjective. Indeed, if one considers the seeds
		$$\Sigma_1=((x_1),\emptyset,[0]), \, \text{ and } \Sigma_2=((x_1),(x_1),[0]),$$
		then it follows from Lemma \ref{lem:injfrozen} that the identity morphism $\mathcal F_{\Sigma_1} = \Q(x_1) \fl \Q(x_1) = \mathcal F_{\Sigma_2}$ induces a rooted cluster morphism $f:\Z[x_1] = \mathcal A(\Sigma_1) \fl \mathcal A(\Sigma_2) = \Z[x_1, \frac{2}{x_1}]$. If $\Sigma_3$ is another seed and $g,h: \mathcal A(\Sigma_2) \fl \mathcal A(\Sigma_3)$ are such that $hf = gf$, then $hf(x_1) = gf(x_1)$ so that $h(x_1) = g(x_1)$ and as $g$ and $h$ are ring homomorphisms, we also have $h(\frac{2}{x_1}) = g(\frac{2}{x_1})$ so that $h=g$. Therefore $f$ is an epimorphism.

		As $f$ is injective, it follows from Proposition \ref{prop:mono} that it is a monomorphism in $\amr$ and thus, $f$ is an example of a \emph{bimorphism} (that is, both a monomorphism and an epimorphism) in $\amr$ which is not an isomorphism.
	\end{rmq}

	\begin{prop}
		Let $\Sigma,\Sigma'$ be two seeds and $f: \mathcal A(\Sigma) \fl \mathcal A(\Sigma')$ be a surjective rooted cluster morphism. Then~:
		\begin{enumerate}
			\item any $\Sigma'$-admissible sequence lifts to an $(f,\Sigma,\Sigma')$-biadmissible sequence,
			\item $\mathcal X_\Sigma' \subset f(\mathcal X_\Sigma)$.
		\end{enumerate}
	\end{prop}
	\begin{proof}
		Let $y \in \ex'$. According to Lemma \ref{lem:surj}, there exists $x \in \ex$ such that $f(x) = y$ so that the $\Sigma'$-admissible sequence $(y)$ lifts to the $(f,\Sigma,\Sigma')$-biadmissible sequence $(x)$. Now let $(y_1, \ldots, y_l)$ be a $\Sigma'$-admissible sequence. We prove by induction on $l$ that $(y_1, \ldots, y_l)$ lifts to an $(f,\Sigma,\Sigma')$-biadmissible sequence. If $l=1$, this follows from the above discussion. Otherwise, there exists $y \in \x'$ such that 
		$$y_l = \mu_{y_{l-1}} \circ \cdots \circ \mu_{y_1}(y).$$
		By the induction hypothesis, $(y_1, \ldots, y_{l-1})$ lifts to an $(f,\Sigma,\Sigma')$-biadmissible sequence $(x_1, \ldots, x_{l-1})$ and since $f$ satisfies \AM 3, we get
		$$y_l = \mu_{f(x_{l-1})} \circ \cdots \circ \mu_{f(x_1)}(x) = f(\mu_{x_{l-1}} \circ \cdots \circ \mu_{x-1}(x))$$
		where $x$ lifts $y$ in $\ex$. Therefore, if $x_l = \mu_{x_{l-1}} \circ \cdots \circ \mu_{x-1}(x)$, the sequence $(x_1, \ldots, x_l)$ lifts $(y_1, \ldots, y_l)$. And moreover, $y_l \in f(\mathcal X_\Sigma)$, which proves the corollary.
	\end{proof}

	\begin{corol}
		Let $\Sigma_1$ and $\Sigma_2$ be two seeds and $f: \mathcal A(\Sigma_1) \fl \mathcal A(\Sigma_2)$ be a surjective rooted cluster morphism. Assume that $\Sigma_1$ is of finite cluster type. Then $\Sigma_2$ is of finite cluster type.
	\end{corol}

	\subsection{Subseeds and surjective morphisms}\label{ssection:surjsubseed}
		Let $\Sigma = (\x,\ex,B)$ be a seed and let $\x' \subset \x$ be a subset. We set $\Sigma' = \Sigma_{|\x'}$ the corresponding full subseed (see Definition \ref{defi:subseed}). We consider the surjective ring homomorphism
		$$\res_{\Sigma,\Sigma'}:\left\{\begin{array}{rcl}
			\mathcal F_{\Sigma} & \fl & \mathcal F_{\Sigma'} \\
			x & \mapsto & x \text{ if }x \in \x',\\
			x & \mapsto & 1 \text{ if }x \in \x \setminus \x'.\\
		\end{array}\right.$$

		\begin{prop}\label{prop:res}
			$\res_{\Sigma,\Sigma'}$ induces an ideal, surjective, rooted cluster morphism $\mathcal A(\Sigma) \fl \mathcal A(\Sigma')$.
		\end{prop}
		\begin{proof}
			In order to simplify notations, we set $\res = \res_{\Sigma,\Sigma'}$ and $\Sigma'= (\x',\ex',B')$. Let $x \in \ex$. Then $\res(x) \in \ex'$ if and only if $x \in \x'$, that is, if and only if $x \in \ex'$. Moreover, 
			$$
				\res(\mu_{x,\Sigma}(x)) 
					 = \res \left(\frac 1x \left( \prod_{\substack{b_{xy}>0 \\ y \in \x}} y^{b_{xy}} + \prod_{\substack{b_{xy}<0 \\ y \in \x}} y^{-b_{xy}} \right)\right) 
					 = \frac 1x \left( \prod_{\substack{b_{xy}>0 \\ y \in \x'}} y^{b_{xy}} + \prod_{\substack{b_{xy}<0 \\ y \in \x'}} y^{-b_{xy}} \right) \\
					 = \mu_{x,\Sigma'}(x).$$
			Moreover, it is easily seen that $\mu_x(B)[\x']= \mu_x(B[\x'])$. By induction, we see that every sequence $(x_1, \ldots, x_n)$ which is $\Sigma$-admissible is $(\res,\Sigma,\Sigma')$-biadmissible and that $\res$ commutes with biadmissible mutations. Therefore $\res$ is a rooted cluster morphism. 

			Moreover, $\res(\mathcal A(\Sigma)) \subset \mathcal A(\Sigma')$ and $\Sigma' = \Sigma_{|\x'} = \res(\Sigma)$ so that Lemma \ref{lem:image} implies that $\res$ is ideal. Finally, by definition of $\Sigma'$, every $\Sigma'$-admissible sequence $(x_1, \ldots, x_n)$ lifts to a biadmissible sequence and therefore, $\mathcal A(\Sigma') \subset \res(\mathcal A(\Sigma))$ so that $\res$ is surjective.
		\end{proof}

	\subsection{Specialisations}
		It is well-known that specialising frozen variables to 1 allows one to realise coefficient-free cluster algebras from cluster algebras of geometric type, see for instance \cite{cluster4}. In this subsection, we study the slightly more general case where an arbitrary cluster variable, frozen or not, is specialised to an integer (which can essentially be assumed to be 1). If the considered cluster variable is frozen, then one finds natural surjective rooted cluster morphisms. More surprisingly, as we prove in certain cases (and expect in general), specialising an \emph{exchangeable} cluster variable to 1 also leads to rooted cluster morphisms.

		We start with a technical lemma~:
		\begin{lem}\label{lem:monome}
			Let $\mathcal A$ be a rooted cluster algebra and let $\x$ be a cluster in $\mathcal A$. Let $m = \prod_{x \in \x} x^{d_x}$ be a Laurent monomial in the variables in $\x$, with $d_x \in \Z$ for any $x \in \x$. Then the following conditions are equivalent~:
			\begin{enumerate}
				\item $m$ is an element in $\mathcal A$~;
				\item $d_x \geq 0$ for any $x \in \x$~;
				\item $m$ is a monomial in $\x$.
			\end{enumerate}
		\end{lem}
		\begin{proof}
			Let $\Sigma$ be a seed containing the cluster $\x$. It is clear that the second and third assertions are equivalent and that the second implies the first one. Therefore, we only have to prove that the first one implies the second one. Assume that there exists some $x \in \x$ such that $d_x < 0$. Because the elements in $\mathcal A$ are Laurent polynomials in the exchangeable variables of $\x$ with polynomial coefficients in the frozen variables of $\x$ (see for instance \cite[Proposition 11.2]{cluster2}), if $d_x<0$, then the variable $x$ is necessarily exchangeable. Let thus $\Sigma'=(\x',\ex',B')=\mu_x \Sigma$ with $\x'= (\x \setminus \ens{x}) \sqcup \ens{x'}$. Then the expansion of $m$ in $\Sigma'$ is
			$$ m = \prod_{y \in \x \setminus x} y^{d_y} \left( \frac{x'}{\prod_{b_{xz}>0} z^{b_{xz}} + \prod_{b_{xz}<0} z^{-b_{xz}}} \right)^{-d_x}.$$
			In particular, $m$ is not a Laurent polynomial in the cluster $\x'$ and thus, according to the Laurent phenomenon (see \cite{cluster1}), $m$ does not belong to $\mathcal A$. 
		\end{proof}

		Let $\Sigma = (\x,\ex,B)$ and let $x \in \x$. We denote by $\Sigma \setminus \ens{x}$ the seed $\Sigma_{|\x \setminus \ens x} = (\x', \ex', B')$ where $\x' = \x \setminus \ens x$, $\ex' = \ex \setminus \ens x$ and $b'_{yz} = b_{yz}$ for any $y,z \neq x$ in $\x'$. 

		\begin{defi}[Simple specialisation]
			Let $n \in \Z$. The \emph{simple specialisation of $x$ to $n$} is the ring homomorphism~:
			$$\sigma_{x,n} : \left\{\begin{array}{rcl}
				\mathcal F_{\Sigma} & \fl & \mathcal F_{\Sigma \setminus \ens{x}}\\
				x & \mapsto & n, \\
				z & \mapsto & z \text{ if } z \in \x \setminus \ens x.
			\end{array}\right.$$
		\end{defi}

		The following lemma shows that, except degenerate cases, the only value to which we can specialise a (single) cluster variable is 1.
		\begin{lem}\label{lem:valuesspe}
			Let $\Sigma = (\x,\ex,B)$ be a seed, let $x \in \x$ and let $n \in \Z$. Assume that $\sigma_{x,n}$ induces a ring homomorphism $\mathcal A(\Sigma) \fl \mathcal A(\Sigma \setminus \ens{x})$. If there exists some $y \in \ex$ such that $b_{xy} \neq 0$, then $n \in \ens{-1,1}$. If there exists some $y \in \ex$ such that $b_{xy} \in 2 \Z +1$, then $n=1$.
		\end{lem}
		\begin{proof}
			Let $x \in \x$, $\x' = \x \setminus \ens x$ and $\Sigma' = \Sigma \setminus \ens x$. Assume that there exists $y \in \ex$ such that $b_{xy} \neq 0$. Without loss of generality, we assume that $b_{xy} >0$. Then
			\begin{align*}
				\sigma(\mu_{y,\Sigma}(y))
					& = \sigma \left( \frac 1 y \left( \prod_{b_{zy>0}}z^{b_{zy}} + \prod_{b_{yz}>0}z^{b_{yz}} \right)\right) \\
					& = \frac 1 y \left( n^{b_{xy}} \prod_{\substack{b_{zy>0}\\z \neq x}}z^{b_{zy}} + \prod_{\substack{b_{yz>0}\\z \neq x}}z^{b_{yz}} \right) \\
					& = \frac 1 y \left( n^{b_{xy}} \prod_{\substack{b'_{zy>0}}}z^{b'_{zy}} + \prod_{\substack{b'_{yz>0}}}z^{b'_{yz}} \right).
			\end{align*}
			But
			$$\mu_{y,\Sigma'}(y) = \frac 1 y \left( \prod_{\substack{b'_{zy>0}}}z^{b_{zy}} + \prod_{\substack{b'_{yz>0}}}z^{b_{yz}} \right).$$
			Thus, if $\sigma$ induces a ring homomorphism between the rooted cluster algebras, we get $\sigma(\mu_{y,\Sigma}(y)) \in \mathcal A(\Sigma \setminus \ens{x})$ and so the difference $\sigma(\mu_{y,\Sigma}(y)) - n^{b_{xy}}\mu_{y,\Sigma'}(y)$ also belongs to $\mathcal A(\Sigma \setminus \ens{x})$.

			But
			$$\sigma(\mu_{y,\Sigma}(y)) - n^{b_{xy}}\mu_{y,\Sigma'}(y) = \frac{1}{y} (1-n^{b_{xy}}) \prod_{b'_{yz}>0} z^{b'_{yz}}$$
			is a Laurent monomial in the cluster $\x'$ such that the exponent of $y$ is $-1 <0$, with $y \in \ex$. Therefore, it follows from \ref{lem:monome} that necessarily $1-n^{b_{xy}} = 0$, that is, $n^{b_{xy}} =1$ and thus $n \in \ens{-1,1}$. If moreover $b_{xy}$ is odd, we necessarily have $n=1$.
		\end{proof}

		\begin{exmp}
			We exhibit an example where a simple specialisation to $-1$ does not induce a map at the level of the corresponding cluster algebras. Consider the cluster algebras of respective types $A_3$ and $A_2$ associated with the coefficient-free seeds
			$$\Sigma = \left((x_1,x_2,x_3), \xymatrix{1 \ar[r] & 2 & \ar[l] 3}\right)
			\text{ and } 
			\Sigma' = \left((x_1,x_2),\xymatrix{1 \ar[r] & 2 }\right)$$
			and consider the simple specialisation $\sigma=\sigma_{x_3,-1}$ of $x_3$ to $-1$.

			The image under $\sigma$ of the cluster variable $\frac{1+x_1x_3}{x_2}$ is $\frac{1-x_1}{x_2}$. But if $\frac{1-x_1}{x_2}$ is in $\mathcal A(\Sigma')$, then, since $\frac{1+x_1}{x_2}$ is in $\mathcal A(\Sigma')$, we get $\frac{2}{x_2} \in \mathcal A(\Sigma')$. Then the expansion of $\frac{2}{x_2}$ in the cluster $(x_1,x_2')$ of the seed $\mu_2(\Sigma')$ is $2\frac{x_2'}{1+x_1}$ which is not a Laurent polynomial, a contradiction.
		\end{exmp}
		
		The following lemma shows that the study of simple specialisations in -1 can be reduced to the study of simple specialisations to 1.
		\begin{lem}
			Let $\Sigma = (\x,\ex,B)$ be a seed and let $x \in \ex$ be such that $b_{xy} \in 2\Z$ for any $y \in \ex$. Then $\sigma_{x,-1}$ is a rooted cluster morphism if and only if $\sigma_{x,1}$ is. Moreover, in this case $\sigma_{x,-1}(a)  = \pm \sigma_{x,1}(a)$ for any $a \in \mathcal X_\Sigma$.
		\end{lem}
		\begin{proof}
			Both $\sigma_{x,1}$ and $\sigma_{x,1}$ satisfy \AM 1 and \AM 2. Now we observe that a sequence of variables is $(\sigma_{x,-1},\Sigma,\Sigma \setminus \ens{x})$-biadmissible if and only if it is disjoint from $x$, which is also the condition for this sequence to be $(\sigma_{x,1},\Sigma,\Sigma \setminus \ens{x})$-biadmissible. Let $a \neq x$ be an exchangeable variable in $\Sigma$.
			Then, since $b_{xy} \in 2\Z$ for any $y \in \ex$, it follows that that $\sigma_{x,-1}(\mu_{a,\Sigma}(a))=\sigma_{x,1}(\mu_{a,\Sigma}(a))$. Moreover, if $\mu_a \Sigma = \Sigma' = (\x',\ex',B')$, then it follows from the mutation rule for exchange matrices that $b'_{xy} \in 2 \Z$ for any $y \in \ex'$. Therefore, by induction, $\sigma_{x,-1}$ satisfies \AM 3 if and only if $\sigma_{x,1}$ does. Hence, it only remains to prove that $\sigma_{x,-1}$ induces a map $\mathcal A(\Sigma) \fl \mathcal A(\Sigma \setminus x)$ if and only if $\sigma_{x,1}$ does. This follows from the fact that $\sigma_{x,-1}(a) \in \ens{-\sigma_{x,1}(a),\sigma_{x,1}(a)}$ for any $a \in \mathcal X_\Sigma$, which is easily proved by induction.
		\end{proof}

		\begin{prop}\label{prop:specialisation}
			Let $\Sigma=(\x,\ex,B)$ be a seed and let $x \in \x$. Then $\sigma_{x,1}$ induces an ideal surjective rooted cluster morphism $\mathcal A(\Sigma) \fl \mathcal A(\Sigma \setminus \ens{x})$ if and only if it induces a ring homomorphism $\mathcal A(\Sigma) \fl \mathcal A(\Sigma \setminus \ens{x})$.
		\end{prop}
		\begin{proof}
			Let $\sigma=\sigma_{x,1}$. Then $\sigma$ clearly satisfies \AM 1 and \AM 2. In order to prove \AM 3, it is enough to notice that a $\Sigma$-admissible sequence $(x_{1}, \ldots, x_{l})$ is $(\sigma,\Sigma,\Sigma \setminus \ens x)$-biadmissible if and only if $x_{k} \neq x$ for any $k$ such that $1 \leq k \leq l$ and to proceed by induction on $l$. It follows that $\sigma$ is a rooted cluster morphism if and only if $\sigma$ induces a ring homomorphism $\mathcal A(\Sigma) \fl \mathcal A(\Sigma \setminus \ens{x})$. Its surjectivity comes from the fact that any admissible sequence in $\Sigma \setminus \ens{x}$ can naturally be lifted to a $(\sigma,\Sigma,\Sigma \setminus \ens x)$-biadmissible sequence. In order to prove that $\sigma$ is ideal, it is enough to observe that $\sigma(\Sigma) = \Sigma \setminus \ens{x}$ so that $\sigma(\mathcal A(\Sigma)) = \mathcal A(\Sigma \setminus \ens x) = \mathcal A(\sigma(\Sigma))$.
		\end{proof}

	\subsection{Specialisations for cluster algebras from surfaces}\label{ssection:specialisationsurfaces}
		In this section, we prove that simple specialisations induce surjective rooted cluster morphisms for cluster algebras associated with surfaces.
		
		Given a marked surface $(S,M)$ and an internal arc $\gamma$ in $(S,M)$, we denote by $d_\gamma(S,M)$ the (non-necessarily connected) marked surface obtained by cutting $(S,M)$ along $\gamma$. The arc $\gamma$ induces in $d_\gamma (S,M)$ two new boundary arcs which we denote by $\gamma_1$ and $\gamma_2$. Now if $T$ is a triangulation of $(S,M)$, then 
		$$d_\gamma T = (T \setminus \ens{\gamma}) \sqcup \ens{\gamma_1, \gamma_2}$$
		is a triangulation of $d_\gamma(S,M)$.

		Figures \ref{fig:cutting1} and \ref{fig:cutting2} present examples of cuttings of surfaces along an arc.

		\begin{center}
			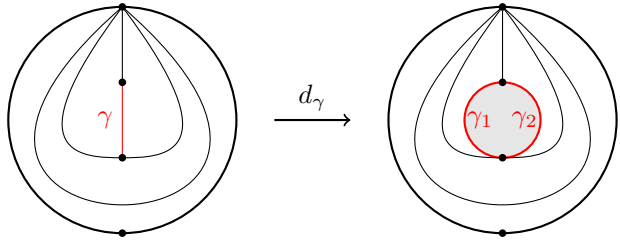
\begin{figure}[htb]
				\begin{tikzpicture}[scale = .5]

					\draw[red]  (0,1) -- (0,-1);
					\fill[red] (0,0) node [left] {$\gamma$};
				
					\draw[thick,->] (4,0) -- (6,0);
					\fill (5,0) node [above] {$d_\gamma$};

					\fill[gray!20] (10,0) circle (1);
					\draw[thick,red] (10,0) circle (1);
					\fill[red] (10,0) node [left] {$\gamma_1$};
					\fill[red] (10,0) node [right] {$\gamma_2$};

					\foreach \x in {0,10}
					{
						\draw[thick] (0+\x,0) circle (3);
						\fill (0+\x,3) circle (.1);
						\fill (0+\x,-3) circle (.1);

						\fill (0+\x,1) circle (.1);
						\fill (0+\x,-1) circle (.1);

						\draw (0+\x,3) .. controls (-8+\x,-4) and (8+\x,-4) .. (0+\x,3);
						\draw (0+\x,3) .. controls (-3+\x,-1) and (-1+\x,-1) .. (0+\x,-1);
						\draw (0+\x,3) .. controls (3+\x,-1) and (1+\x,-1) .. (0+\x,-1);
						\draw (0+\x,3) -- (0+\x,1);
					}
				\end{tikzpicture}
				\caption{Cutting a disc with two punctures to get an annulus.}\label{fig:cutting1}
			\end{figure}
		\end{center}

		\begin{center}
			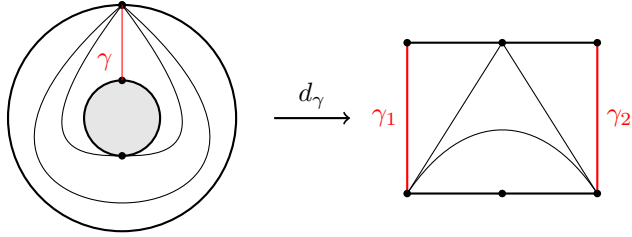
\begin{figure}[htb]
				\begin{tikzpicture}[scale = .5]
					\foreach \x in {0}
					{
						\fill[gray!20] (0+\x,0) circle (1);
						\draw[thick] (0+\x,0) circle (1);

						\draw[thick] (0+\x,0) circle (3);
						\fill (0+\x,3) circle (.1);
						\fill (0+\x,3) circle (.1);

						\fill (0+\x,1) circle (.1);
						\fill (0+\x,-1) circle (.1);

						\draw (0+\x,3) .. controls (-8+\x,-4) and (8+\x,-4) .. (0+\x,3);
						\draw (0+\x,3) .. controls (-3+\x,-1) and (-1+\x,-1) .. (0+\x,-1);
						\draw (0+\x,3) .. controls (3+\x,-1) and (1+\x,-1) .. (0+\x,-1);
						\draw[red] (0+\x,3) -- (0+\x,1);
					}

					\draw[thick,->] (4,0) -- (6,0);
					\fill (5,0) node [above] {$d_\gamma$};

					\foreach \x in {10}
					{
			
						\draw[thick,red] (-2.5+\x,-2) -- (-2.5+\x,2);
						\draw[thick,red] (2.5+\x,-2) -- (2.5+\x,2);
						\draw[thick] (-2.5+\x,-2) -- (2.5+\x,-2);
						\draw[thick] (-2.5+\x,2) -- (2.5+\x,2);
			
						\fill (-2.5+\x,2) circle (.1);
						\fill (-2.5+\x,-2) circle (.1);
						\fill (2.5+\x,2) circle (.1);
						\fill (2.5+\x,-2) circle (.1);

						\fill (\x,2) circle (.1);
						\fill (\x,-2) circle (.1);

						\draw (-2.5+\x,-2) -- (\x,2);
						\draw (2.5+\x,-2) -- (\x,2);
						\draw (-2.5+\x,-2) .. controls (\x-1,.25) and (\x+1,.25) .. (2.5+\x,-2);
					}
					\fill[red] (0,1.5) node [left] {$\gamma$};

					\fill[red] (7.5,0) node [left] {$\gamma_1$};
					\fill[red] (12.5,0) node [right] {$\gamma_2$};
				\end{tikzpicture}
				\caption{Cutting an annulus to get a disc without punctures.}\label{fig:cutting2}
			\end{figure}
		\end{center}
		
		Note that the seed $\Sigma_{d_\gamma T}$ differs from the seed $\Sigma_T \setminus \ens{x_\gamma}$. However, we have $\Sigma_{d_\gamma T}[T \setminus \ens{\gamma_1,\gamma_2}] = \Sigma_T[T \setminus \ens{\gamma}]$.

		\begin{theorem}\label{theorem:specialisationsurfaces}
			Let $(S,M)$ be a marked surface, $T$ be a triangulation of $(S,M)$ and $\Sigma_T$ be the seed associated with $T$. Then for any $\gamma \in T$ which does not enclose a degenerate marked surface, the simple specialisation of $x_\gamma$ to 1 induces an ideal surjective rooted cluster morphism in $\mathcal A(\Sigma_T) \fl \mathcal A(\Sigma_T \setminus \ens{x_\gamma})$.
		\end{theorem}
		\begin{proof}
			If $\gamma$ is a boundary arc, then $x_\gamma$ is a coefficient and the result is well-known, see \cite{cluster4}. We may thus assume that $\gamma$ is an internal arc. 

			According to Proposition \ref{prop:specialisation}, as $\sigma=\sigma_{x_\gamma,1}$ is a ring homomorphism, in order to prove that $\sigma$ is a rooted cluster morphism, it is enough to prove that the image of $\sigma$ is contained in $\mathcal A(\Sigma_T \setminus \ens {x_\gamma})$. For this we only need to prove that $\sigma(x_\eta) \in \mathcal A(\Sigma_T \setminus \ens {x_\gamma})$ for any (possibly tagged) arc $\eta$ in $(S,M)$. For the sake of simplicity we only prove it for an untagged arc $\eta$. The case of tagged arcs is a straightforward adaptation.
	
			Let $\eta$ be an arc in $(S,M)$. Resolving the intersections of $\eta$ with $\gamma$ (using for instance the resolutions described in \cite{DP:nonorientable} or more generally the skein relations described in \cite{MW:resolutions}), we can write $x_\gamma x_\eta$ as a linear combination of products of $x_\theta$ where $\theta$ runs over a family of curves which do not intersect $\gamma$. 

			Every curve which does not cross $\gamma$ induces a curve in the surface $d_\gamma (S,M)$. Let $\tau$ denote the specialisation of $x_{\gamma_1}$ and $x_{\gamma_2}$ to 1. Then, as $\gamma_1$ and $\gamma_2$ are boundary arcs, $\tau$ is a rooted cluster morphism from $\mathcal A(\Sigma_{d_\gamma T})$ to $\mathcal A(\Sigma_{d_\gamma T} \setminus \ens{x_{\gamma_1},x_{\gamma_2}}) = \mathcal A(\Sigma_{T} \setminus \ens{x_\gamma})$. Moreover, for any arc $\theta$ in $(S,M)$ which does not cross $\gamma$, we have $\sigma(x_\theta) = \tau(x_\theta)$.

			Therefore, $\sigma(x_\eta) = \sigma(x_\gamma x_\eta)$ is a linear combination of $\tau(x_\theta)$ where $\theta$ runs over a family of curves which do not intersect $\gamma$. Now for each such curve, $x_\theta$ is an element of the cluster algebra $\mathcal A(d_\gamma T)$ and thus $\tau(x_\theta)$ is an element of the cluster algebra $\mathcal A(\Sigma_T \setminus \ens{x_\gamma})$.

			The fact that $\sigma$ is surjective is clear since any admissible sequence for $\Sigma_T \setminus \ens{x_\gamma}$ lifts to an admissible sequence for $\Sigma_T$. The fact that it is ideal comes from the fact that $\sigma(\Sigma_T) = \Sigma_T \setminus \ens{x_\gamma}$.
		\end{proof}

		\begin{exmp}
			Consider the once-punctured torus $\mathbb T_1$ and fix a triangulation $T$ of $\mathbb T_1$. It has three arcs which we denote by 1,2 and 3 and which we show as follows in the universal cover of $\mathbb T_1$~:
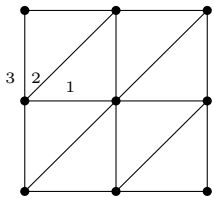
			\begin{center}
				\begin{figure}[htb]
					\begin{tikzpicture}[scale = .6]

						\fill (1,2) node [above] {\tiny 1};
						\fill (.25,2.5) node {\tiny 2};
						\fill (0,2.5) node [left] {\tiny 3};

						\foreach \x in {0}
						{
							\foreach \y in {0}
							{
								\draw (\x,\y) -- (\x+4,\y);
								\draw (\x,\y+2) -- (\x+4,\y+2);
								\draw (\x,\y+4) -- (\x+4,\y+4);

								\draw (\x,\y) -- (\x,\y+4);
								\draw (\x+2,\y) -- (\x+2,\y+4);
								\draw (\x+4,\y) -- (\x+4,\y+4);

								\draw (\x,\y) -- (\x+4,\y+4);
								\draw (\x,\y+2) -- (\x+2,\y+4);
								\draw (\x+2,\y) -- (\x+4,\y+2);

								\fill (\x,\y) circle (.1);
								\fill (\x+2,\y) circle (.1);
								\fill (\x+4,\y) circle (.1);
								\fill (\x,\y+2) circle (.1);
								\fill (\x+2,\y+2) circle (.1);
								\fill (\x+4,\y+2) circle (.1);
								\fill (\x,\y+4) circle (.1);
								\fill (\x+2,\y+4) circle (.1);
								\fill (\x+4,\y+4) circle (.1);
							}
						}

					\end{tikzpicture}
				\end{figure}
			\end{center}
			The seed corresponding to this triangulation is the coefficient-free seed 
			$$\Sigma_T = \left((x_1,x_2,x_3),\left[\begin{array}{rrr}0 & 2 & -2 \\ -2 & 0 & 2 \\ 2 & -2 &  0\end{array}\right]\right).$$

			Now, an arbitrary cluster variable in $\mathcal A(\mathbb T_1)$ corresponds to a certain arc. In this example we choose for instance the curve $\eta$ shown below. Let us cut $\mathbb T_1$ along the arc 1, which we show dashed. Resolving intersections between $\eta$ and 1, and applying skein relations we get $x_\eta x_1 = x_2x_3 + x_2 x_\theta$ where $\theta$ does not intersect the arc 1. 

			
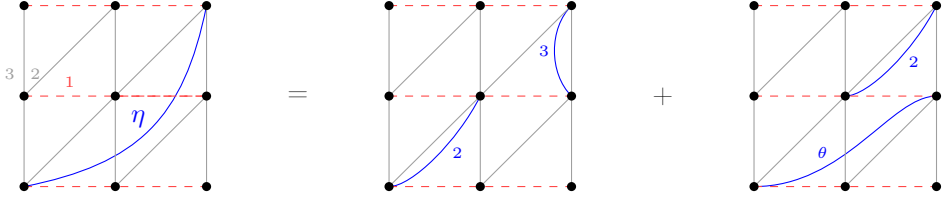
\begin{center}
				\begin{tikzpicture}[scale = .6]

					\draw[dashed,red] (2,2) -- (4,2);

					\draw[blue] (0,0) .. controls (2.5,.5) and (3.5,1.5) .. (4,4);

					\fill[dashed,red!75] (1,2) node [above] {\tiny 1};
					\fill[gray!75] (.25,2.5) node {\tiny 2};
					\fill[gray!75] (0,2.5) node [left] {\tiny 3};

					\fill[blue] (2.5,1.5) node {$\eta$};

					\draw[blue] (8,0) .. controls (8.5,0) and (9.5,1) .. (10,2);
					\draw[blue] (12,2) .. controls (11.5,2.5) and (11.5,3.5) .. (12,4);

					\fill[blue] (9.5,.75) node {\tiny $2$};

					\fill[blue] (11.4,3) node {\tiny $3$};

					\draw[blue] (16,0) .. controls (18,0) and (19,2) .. (20,2);
					\draw[blue] (18,2) .. controls (18.5,2) and (19.5,3) .. (20,4);

					\fill[blue] (17.5,.75) node {\tiny $\theta$};

					\fill[blue] (19.5,2.75) node {\tiny $2$};
				
					\fill (6,2) node {=};
					\fill (14,2) node {+};

					\foreach \x in {0,8,16}
					{
						\foreach \y in {0}
						{
							\draw[dashed,red!75] (\x,\y) -- (\x+4,\y);
							\draw[dashed,red!75] (\x,\y+2) -- (\x+4,\y+2);
							\draw[dashed,red!75] (\x,\y+4) -- (\x+4,\y+4);

							\draw[gray!75] (\x,\y) -- (\x,\y+4);
							\draw[gray!75] (\x+2,\y) -- (\x+2,\y+4);
							\draw[gray!75] (\x+4,\y) -- (\x+4,\y+4);

							\draw[gray!75] (\x,\y) -- (\x+4,\y+4);
							\draw[gray!75] (\x,\y+2) -- (\x+2,\y+4);
							\draw[gray!75] (\x+2,\y) -- (\x+4,\y+2);

							\fill (\x,\y) circle (.1);
							\fill (\x+2,\y) circle (.1);
							\fill (\x+4,\y) circle (.1);
							\fill (\x,\y+2) circle (.1);
							\fill (\x+2,\y+2) circle (.1);
							\fill (\x+4,\y+2) circle (.1);
							\fill (\x,\y+4) circle (.1);
							\fill (\x+2,\y+4) circle (.1);
							\fill (\x+4,\y+4) circle (.1);
						}
					}
				\end{tikzpicture}
			\end{center}

			We can easily compute 
			$$x_\theta = \frac{x_1^2+x_2^2}{x_3}$$
			so that $\sigma_{x_1,1}(x_\theta) = \frac{1+x_2^2}{x_3} \in \mathcal A(\Sigma \setminus \ens{x_1})$. Let us also give a geometric argument.

			Cutting $\mathbb T_1$ along the arc 1, we get the annulus $C_{1,1}$ with one marked point on each boundary component and $\ens{2,3}$ together with the two boundary arcs $b$ and $b'$ induce a triangulation of $C_{1,1}$. The arc $\theta$ in $\mathbb T_1$ induces an arc in $C_{1,1}$ and we denote by $x'_{\theta}$ the corresponding cluster variable in $\mathcal A(C_{1,1})$. 

			\begin{center}
				\begin{tikzpicture}[scale = .6]

					\draw[blue] (0,0) .. controls (2,0) and (3,2) .. (4,2);

					\fill[blue] (1.5,.75) node {\tiny $\theta$};

					\fill[red] (1,2) node [above] {\tiny $b$};
					\fill[red] (1,0) node [below] {\tiny $b'$};

					\fill[red] (3,2) node [above] {\tiny $b$};
					\fill[red] (3,0) node [below] {\tiny $b'$};

					\foreach \x in {0}
					{
						\foreach \y in {0}
						{
							\draw[thick,red!75] (\x,\y) -- (\x+4,\y);
							\draw[thick,red!75] (\x,\y+2) -- (\x+4,\y+2);

							\draw[gray!75] (\x,\y) -- (\x,\y+2);
							\draw[gray!75] (\x+2,\y) -- (\x+2,\y+2);
							\draw[gray!75] (\x+4,\y) -- (\x+4,\y+2);

							\draw[gray!75] (\x,\y) -- (\x+2,\y+2);
							\draw[gray!75] (\x+2,\y) -- (\x+4,\y+2);

							\fill (\x,\y) circle (.1);
							\fill (\x+2,\y) circle (.1);
							\fill (\x+4,\y) circle (.1);
							\fill (\x,\y+2) circle (.1);
							\fill (\x+2,\y+2) circle (.1);
							\fill (\x+4,\y+2) circle (.1);
						}
					}
				\end{tikzpicture}
			\end{center}			

			Then, it follows directly from the various expansion formulae for cluster variables associated with arcs (see for instance \cite{ST:unpunctured,MSW:positivity} or \cite{ADSS:strings}) that $x_\theta$ is given by $x'_\theta$ where the variables $x_b$ and $x_{b'}$ corresponding to the two boundary components are identified with $x_1$. Indeed, a direct computation gives
			$$x'_\theta = \frac{x_bx_{b'}+x_2^2}{x_3}.$$

			In particular, if we specialise $x_b,x_b'$ and $x_1$ to 1, we still get equality. But the cluster algebra associated with $C_{1,1}$ whose frozen variables are specialised to 1 is nothing but the cluster algebra associated with the seed  $\mathcal A(\Sigma_T \setminus \ens{x_1})$.
		\end{exmp}

	\subsection{Specialisations in acyclic cluster algebras}
		We now prove an analogue result for acyclic cluster algebras using techniques coming from additive categorifications of cluster algebras. We refer the interested reader to \cite{Keller:survey} and references therein.

		Philosophically speaking, the proof of the following theorem follows the lines of the proof of Theorem \ref{theorem:specialisationsurfaces} where, in the spirit of \cite{BZ:clustercatsurfaces}, objects in categories should be thought as curves in a surface, extensions of objects as intersections of the corresponding curves and ``Hall products'' as skein relations.

		\begin{theorem}\label{theorem:acyclicspe}
			Let $\Sigma=(\x,\ex,B)$ be a seed which is mutation-equivalent to a finite skew-symmetric acyclic seed. Then for any $x \in \x$, we have $\sigma_{x,1}(\mathcal A(\Sigma)) \subset \mathcal A(\Sigma \setminus \ens x) \otimes_{\Z} \Q$.
		\end{theorem}
		\begin{proof}
			Without loss of generality we can assume that $\Sigma = \overline \Sigma$ is simplified. Let $Q$ be an acyclic quiver corresponding to the exchange matrix of a seed which is mutation-equivalent to $\Sigma$. 

			Let $\CC$ be the cluster category of $Q$ over an algebraically closed field $\k$ introduced in \cite{BMRRT}. Its suspension functor is denoted by $[1]$. Let $T = \bigoplus_{y \in \x} T_y$ be a cluster-tilting object in $\CC$ corresponding to the seed $\Sigma$, see \cite{CK2,FK}.

			Let $T_x^{\bot}$ be the full subcategory of $\CC$ formed by the objects $V$ such that $\Hom_{\CC}(T_x,V)=0$. Then it follows from \cite{IY} that $\CC' = T_x^{\bot}/T_x[1]$ is a Hom-finite 2-Calabi-Yau category and that $T' = \bigoplus_{y \in \x \setminus \ens x} T_y$ is a cluster-tilting object in $\CC'$. The quiver $Q_{T'}$ of the endomorphism ring of $T'$ is a full subquiver of the quiver $Q_T$ of the endomorphism ring of $T$. By assumption $Q_T$ is mutation-acyclic. Thus, it follows from \cite{BMR2} (see also \cite{Warkentin:mutationacyclic}) that $Q_{T'}$ is also mutation-acyclic and therefore $\CC'$ is triangle equivalent to a cluster category, see \cite{KR:acyclic}.

			We denote by $X_?$ (or $X'_?$, respectively) the cluster character associated to $T$ on $\CC$ (or to $T'$ on $\CC'$, respectively), see \cite{Palu}. Since $\CC$ and $\CC'$ are cluster categories associated to acyclic quivers, it follows from \cite[Theorem 4]{GLS:generic} that the map $X_?$ (or $X'_?$, respectively) has values in the cluster algebra $\mathcal A(\Sigma)$ (or $\mathcal A(\Sigma \setminus \ens x)$, respectively).

			Let $\sigma = \sigma_{x,1}$ be the simple specialisation of $x$ to 1. Let $M$ be an object in $\CC$. We prove by induction on the dimension of $\Hom_{\CC}(T_x,M)$ that $\sigma(X_M)$ is a finite $\Q$-linear combination of $X'_Y$ where $Y$ runs over the objects of $\CC'$. 

			Assume first that $\Hom_{\CC}(T_x,M) = 0$, that is, $M$ belongs to $T_x^{\bot}$. Then $M$ can be decomposed as $\overline M \oplus T_x[1]^m$ for some $m \geq 1$ where $\overline M$ has no direct summand isomorphic to $T_x[1]$. Therefore, 
			$$\sigma(X_M) = \sigma( X_{\overline M \oplus T_x[1]^m}) = \sigma(X_{\overline M} X_{T_x[1]}^m) = \sigma(X_{\overline M})\sigma(X_{T_x[1]}^m) = \sigma(X_{\overline M})\sigma(x^m) = \sigma(X_{\overline M}).$$
			Since $M$ belongs to $T_x^{\bot}$, the object $\overline M$ belongs to $T_x^{\bot}/T_x[1]$ and is thus identified with an object in $\CC'$. Now it follows easily from the definition of the cluster characters that $\sigma(X_{\overline M}) = X'_{\overline M}$. Therefore, $\sigma(X_M)$ belongs to $\mathcal A(\Sigma \setminus \ens x)$.

			Assume now that $\Hom_{\CC}(T_x,M) \neq 0$. Therefore, 
			$$\Ext^1_{\CC}(T_x[1],M) = \Hom_{\CC}(T_x[1],M[1]) \simeq \Hom_{\CC}(T_x,M) \neq 0.$$
			Then it follows from \cite{Palu:multiplication} that the product $(\dim \Ext^1_{\CC}(T_x[1],M) X_{T_x[1]} X_M)$ is a $\Z$-linear combination of $X_Y$'s where $Y$ runs over middle terms of non-split triangles of the form 
			\begin{equation}\label{eq:MTx}
				M \fl Y \fl T_x[1] \fl M[1]
			\end{equation}
			or of the form
			\begin{equation}\label{eq:TxM}
				T_x[1] \fl Y \xrightarrow{a} M \xrightarrow{b} T_x[2].
			\end{equation}
			In other words, we have 
			\begin{equation}\label{eq:expansion}
				xX_M = \sum_Y n_Y X_Y
			\end{equation}
			where $(n_Y) \subset \Q$ is finitely supported on a set of isoclasses of objects $Y$ in $\CC$ such that there exists triangles of the form \eqref{eq:MTx} or \eqref{eq:TxM}.

			We claim that $\dim \Hom_{\CC}(T_x,Y) < \dim \Hom_{\CC}(T_x,M)$ for any $Y$ such that $n_Y \neq 0$. Indeed, assume first that $Y$ is such that there exists a triangle of the form \eqref{eq:MTx}, that is, 
			$$T_x \xrightarrow{\alpha} M \xrightarrow{\beta} Y \fl T_x[1] \fl M[1].$$
			Applying the homological functor $\Hom_{\CC}(T_x,-)$ to this triangle, we get the long exact sequence
			$$\End_{\CC}(T_x) \xrightarrow{\alpha_*} \Hom_{\CC}(T_x,M) \xrightarrow{\beta_*} \Hom_{\CC}(T_x,Y) \fl 0,$$
			so that the post-composition by $\beta$ yields an epimorphism $\Hom_{\CC}(T_x,M) \xrightarrow{\beta_*} \Hom_{\CC}(T_x,Y)$. Since the sequence is exact, we have $\ker(\beta_*) = \im(\alpha_*)$. Moreover, $\End_{\CC}(T_x) \xrightarrow{\alpha_*} \Hom_{\CC}(T_x,M)$ is non-zero since $\alpha_*(\id_{T_x}) = \alpha \neq 0$. Therefore, $\dim \Hom_{\CC}(T_x,Y) < \dim \Hom_{\CC}(T_x,M)$ in this case, as claimed.

			Assume now that $Y$ is such that there exists a triangle of the form \eqref{eq:TxM}, that is, 
			$$T_x[1] \fl Y \xrightarrow{a} M \xrightarrow{b} T_x[2].$$
			Applying the homological functor $\Hom_{\CC}(T_x,-)$ to this triangle, we get the long exact sequence
			$$0 \fl \Hom_{\CC}(T_x,Y) \xrightarrow{a_*} \Hom_{\CC}(T_x,M) \xrightarrow{b_*} \Hom_{\CC}(T_x,T_x[2]).$$
			Therefore, the post-composition by $a$ yields an injection $\Hom_{\CC}(T_x,Y) \xrightarrow{a_*} \Hom_{\CC}(T_x,M)$. In order to prove that the injection is proper, since the sequence is exact, we need to prove that the post-composition $\Hom_{\CC}(T_x,M) \xrightarrow{b_*} \Hom_{\CC}(T_x,T_x[2])$ is non-zero. Since $\CC$ is 2-Calabi-Yau, we have the commutative diagram
			$$\xymatrix{
				\Hom_{\CC}(T_x,M) \ar[d]_{\sim} \ar[r]^{b_*} & \Hom_{\CC}(T_x,T_x[2]) \ar[d]^{\sim} \\
				D\Hom_{\CC}(M,T_x[2]) \ar[r]^{Db^*} & D\End_{\CC}(T_x[2])
			}$$
			where $D = \Hom_{\k}(-,\k)$ is the standard duality and $\End_{\CC}(T_x[2]) \xrightarrow{b^*} \Hom_{\CC}(M,T_x[2])$ is the pre-composition by $b$. Since $b$ is non-zero, $b^*(\id_{T_x[2]}) \neq 0$ and thus $Db^*$ is non-zero so that $b_*$ is non-zero. This proves the claim.

			Therefore, equality \eqref{eq:expansion} allows to write $xX_M$ as a $\Q$-linear combination of elements $X_Y$ where $\dim \Hom_{\CC}(T_x,Y) < \dim \Hom_{\CC}(T_x,M)$. If $\dim \Hom_{\CC}(T_x,Y) = \dim \Ext^1_{\CC}(T_x[1],Y) \neq 0$, we can again write $xX_Y$ as a linear combination of cluster characters of objects for which the dimension is strictly smaller. Proceeding by induction, there exists some $n \geq 1$ such that $x^n X_M$ is a $\Q$-linear combination of elements of the form $X_Y$ where $\dim \Hom_{\CC}(T_x,Y) = 0$. Therefore $\sigma(X_M) = \sigma(x^n X_M)$ is a $\Q$-linear combination of $\sigma(X_Y)$ with $\dim \Hom_{\CC}(T_x,Y) = 0$. Hence, it follows from the previous discussion that $\sigma(X_M)$ belongs to $\mathcal A(\Sigma \setminus \ens x)$.
		\end{proof}

		\begin{rmq}
			In the proof of Theorem \ref{theorem:acyclicspe}, the assumption that $\Sigma$ is mutation-acyclic is only used in order to prove that the cluster characters take their values in the cluster algebras. In fact, the proof shows that for any cluster variable $x$ in a seed $\Sigma$ admitting an additive Hom-finite 2-Calabi-Yau categorification, the map $\sigma_{x,1}$ sends any cluster variable in $\mathcal A(\Sigma)$ to a $\Q$-linear combination of cluster characters associated to objects in a Hom-finite 2-Calabi-Yau categorification of $\Sigma \setminus \ens x$. However, in general these latter cluster characters may not live in the cluster algebra but only in the upper cluster algebra, see for instance \cite[Example 5.6.3]{Plamondon:thesis}. Interactions of simple specialisations with lower and upper bounds of cluster algebras are studied in Subsection \ref{ssection:spegeneral}.
		\end{rmq}

		The following example illustrates the fact appearing in the proofs of Theorems \ref{theorem:specialisationsurfaces} and \ref{theorem:acyclicspe} that the simple specialisation of an exchangeable variable to 1 does not send a cluster variable to a cluster variable in general, but rather sends it to a linear combination of elements in the cluster algebra. 

		\begin{exmp}\label{exmp:spe}
			Consider the coefficient-free seed $\Sigma = ((x_1,x_2,x_3,x_4),B)$ where
			$$B= \left[\begin{array}{rrrr}
				0 & 0 & 0 & 1 \\
				0 & 0 & 0 & 1 \\
				0 & 0 & 0 & 1 \\
				-1 & -1 & -1 & 0 \\
			\end{array}\right]$$
			is the incidence matrix of the quiver $Q$ of Dynkin type $D_4$ where 4 is a sink.

			We consider the specialisation of $x_4$ to 1 so that we also consider the seed 
			$$\Sigma \setminus \ens{x_4}  = \left( (x_1,x_2,x_3), [0] \right)$$
			which is of type $A_1 \times A_1 \times A_1$. In particular,
			$$\mathcal A(\Sigma \setminus \ens{x_4}) = \Z\left[x_1,\frac{2}{x_1},x_2,\frac{2}{x_2},x_3,\frac{2}{x_3}\right] \subset \Q(x_1,x_2,x_3).$$

			Consider the cluster variable in $\mathcal A(\Sigma)$
			$$x = \frac{1+x_1x_2x_3+3x_4+3x_4^2+x_4^3}{x_1x_2x_3x_4}$$
			which, in the context of \cite{CC,BMRRT}, corresponds to the cluster character of the indecomposable representation of $Q$ with dimension vector $(1111)$.

			Then
			$$\sigma_{x_4,1} (x) = \frac{8}{x_1x_2x_3} +1 = \frac{2}{x_1} \frac{2}{x_2} \frac{2}{x_3}+1$$
			is the sum of 1 and the cluster monomial in $\mathcal A(\Sigma \setminus \ens{x_4})$ corresponding to the cluster character of the semisimple representation with dimension vector $(111)$ of $Q \setminus \ens{4}$.

			Now consider the cluster variable in $\mathcal A(\Sigma)$
			$$z = \frac{1 + 2x_1x_2x_3 + x_1^2x_2^2x_3^2 + 3x_4 + 3x_1x_2x_3x_4 + 3 x_4^2 + x_4^3}{x_1x_2x_3x_4^2}$$
			which corresponds to the cluster character of the indecomposable representation of $Q$ with dimension vector $(1112)$.

			Then
			$$\sigma_{x_4,1} (x) = \frac{2}{x_1}\frac{2}{x_2}\frac{2}{x_3} + 5 + x_1x_2x_3.$$
			is a linear combination of three distinct cluster monomials in $\mathcal A(\Sigma \setminus \ens{x_4})$. 
		\end{exmp}

	\subsection{Simple specialisations in general}\label{ssection:spegeneral}
		In general, we expect that simple specialisations of cluster variables to 1 induce rooted cluster morphisms. 
		\begin{prob}\label{prob:specialisation}
			Let $\Sigma=(\x,\ex,B)$ be a seed and let $x \in \x$. Then does $\sigma_{x,1}$ induce a surjective ideal rooted cluster morphism $\mathcal A(\Sigma) \fl \mathcal A(\Sigma \setminus \ens x)$~?
		\end{prob}

		We now prove that simple specialisations preserve upper and lower bounds of cluster algebras in general.

		Given a seed $\Sigma = (\x,\ex,B)$, and given an element $x \in \ex$, we denote by $\mu_x(\ex)$ the set of exchangeable variables in $\mu_x(\Sigma)$ so that $\Z[\x \setminus \ex][\mu_x(\ex)^{\pm 1}]$ is the set of Laurent polynomials in the exchangeable variables of $\mu_x(\Sigma)$ with polynomial coefficients in the frozen variables of $\mu_x(\Sigma)$.
		We also set 
		$$\ex' = \bigcup_{x \in \ex} \mu_x(\ex)$$
		to be the set of all the exchangeable variables in the seeds obtained from $\Sigma$ by applying exactly one mutation.

		Following \cite{cluster3}, we set~:
		\begin{defi}
			\begin{enumerate}
				\item The \emph{lower bound} of $\mathcal A(\Sigma)$ is 
					$$\mathcal L(\Sigma) = \Z[\x \setminus \ex][\ex \cup \ex'].$$
				\item The \emph{upper bound} of $\mathcal A(\Sigma)$ is
					$$\mathcal U(\Sigma) = \Z[\x \setminus \ex][\ex^{\pm 1}] \cap \bigcap_{x \in \ex} \Z[\x \setminus \ex][\mu_x(\ex)^{\pm 1}].$$ 
			\end{enumerate}
		\end{defi}
			
		These are subalgebras of $\mathcal F_{\Sigma}$ and we always have the inclusions $\mathcal L(\Sigma) \subset \mathcal A(\Sigma) \subset \mathcal U(\Sigma)$.

		\begin{prop}\label{prop:spebounds}
			Let $\Sigma = (\x,\ex,B)$ be a seed and let $x \in \x$. Then~:
			\begin{enumerate}
				\item $\sigma_{x,1}(\mathcal L(\Sigma)) \subset \mathcal L(\Sigma \setminus \ens x)$~;
				\item $\sigma_{x,1}(\mathcal U(\Sigma)) \subset \mathcal U(\Sigma \setminus \ens x)$.
			\end{enumerate}
		\end{prop}
		\begin{proof}
			If $x$ is frozen, the result is clear. Therefore, we fix some exchangeable variable $x$ in $\ex$. In order to simplify the notations, we set $\sigma = \sigma_{x,1}$. 

			Let $z$ be a cluster variable in $\Sigma$. If $z$ is frozen, then $\sigma(z) = z$ so that $\sigma(z)$ is both in the upper and in the lower bounds of $\mathcal A(\Sigma \setminus \ens x)$. Assume now that $z$ is exchangeable. If $z \neq x$, we have
			\begin{align*}
				\sigma(\mu_{z,\Sigma}(z)) 
					& = \sigma\left(\frac{1}{z} \left( \prod_{\substack{b_{yz}>0~;\\ y \in \x}} y^{b_{yz}} + \prod_{\substack{b_{yz}<0~;\\ y \in \x}} y^{-b_{yz}}\right)\right) \\
					& = \frac{1}{z} \left( \prod_{\substack{b_{yz}>0~;\\ y \in \x \setminus \ens x}} y^{b_{yz}} + \prod_{\substack{b_{yz}<0~;\\ y \in \x \setminus \ens x}} y^{-b_{yz}}\right)\\
					& = \mu_{z,\Sigma \setminus \ens x}(z).
			\end{align*}
			If $z=x$, we get  
			$$\sigma(\mu_{x,\Sigma}(x)) = \sigma\left( \frac{1}{x} \left(\prod_{\substack{b_{xy}>0\\ y \in \x}} y^{b_{xy}} + \prod_{\substack{b_{xy}<0\\ y \in \x}} y^{-b_{xy}} \right)\right)  = \prod_{\substack{b_{xy}>0\\ y \in \x \setminus \ens{x}}} y^{b_{xy}} + \prod_{\substack{b_{xy}<0\\ y \in \x \setminus \ens{x}}} y^{-b_{xy}}.$$
			It easily follows that $\sigma(\mathcal L(\Sigma)) \subset \mathcal L(\Sigma \setminus \ens x)$ and $\sigma(\mathcal U(\Sigma)) \subset \mathcal U(\Sigma \setminus \ens x)$.
		\end{proof}

	\subsection{An example of multiple specialisation in zero}\label{ssection:Gr}
		For any $m \geq 4$, we denote by $\Gr_2(m)$ the set of planes in $\C^m$ and let $\C[\Gr_2(m)]$ denote its ring of homogeneous coordinates. It is known that 
		$$\C[\Gr_2(m)] \simeq \mathcal A(\Sigma_m) \otimes_{\Z} \C$$
		where $\Sigma_m$ is the seed constructed in \S \ref{ssection:typeA}, see for instance\cite{Scott:Gr} or \cite[Chapter 2]{GSV:book}. The cluster variables in $\mathcal A(\Pi_m)$ are identified with the Plücker coordinates $x_{k,l}$, with $1 \leq k < l \leq m$ in $\C[\Gr_2(m)]$ in such a way that the Plücker coordinate $x_{k,l}$ corresponds to the arc joining $k$ to $l$ in $\Pi_m$, see [loc. cit.].

		Let $m,m'$ be integers such that $4 \leq m' \leq m$. The choice of an inclusion of $\C^{m'}$ into $\C^{m}$ induces an embedding $\iota:\Gr_2(m') \fl \Gr_2(m)$ and thus an epimorphism of $\C$-algebras~:
		$$\iota^*:\C[\Gr_2(m)] \fl \C[\Gr_2(m')].$$

		If $x'_{k,l}$, with $1 \leq k < l \leq m'$, are the Plücker coordinates on $\C[\Gr_2(m')]$, then the morphism $\iota^*$ is given by
		$$\iota^*(x_{k,l}) = \left\{\begin{array}{ll}
			x'_{k,l} & \text{ if } l \leq m' \\
			0	& \text{ if } l > m'
		\end{array}\right.$$
		for $1 \leq k <l \leq m$.

		\begin{prop}
			There exists a unique rooted cluster morphism $\pi_{m,m'}:\mathcal A(\Sigma_m) \fl \mathcal A(\Sigma_{m'})$ such that $\iota^* = \pi \otimes_{\Z} \1_{\C}$.
		\end{prop}
		\begin{proof}
			Consider the ring homomorphism $\pi_{m,m'}:\mathcal F_{\Sigma_m} \fl \mathcal F_{\Sigma_{m'}}$ acting as $\iota^*$ on the Plücker coordinates. Then $\pi$ defines a ring homomorphisms from $\mathcal A(\Sigma_m)$ to $\mathcal A(\Sigma_{m'})$. Moreover, $\pi_{m,m'}$ satisfies \AM 1 and \AM 2 by construction.

			Exchange relations in $\mathcal A(\Sigma_m)$ are given by
			$$x_{ij}x_{kl} = x_{ik}x_{jl} + x_{il}x_{jk}$$
			for any $i,j,k,l$ such that $1 \leq i<k<j<l \leq m$ and similarly for $\mathcal A(\Sigma_{m'})$. Thus, $\pi$ sends exactly the exchange relations involving only $x_{k,l}$ with $1 \leq i<k<j<l \leq m'$ in $\mathcal A(\Sigma_m)$ to the same exchange relations for $x'_{k,l}$ with $1 \leq i<k<j<l \leq m'$ in $\mathcal A(\Sigma_{m'})$. In other words, $\pi$ commutes with mutations along biadmissible sequences and thus satisfies \AM 3.

			For uniqueness, it is enough to observe that if such a morphism $\pi_{m,m'}$ exists then it necessarily acts as $\iota^*$ on the Plücker coordinates and thus it is unique.
		\end{proof}

\section{Surgery}\label{section:surgery}
	In this section we introduce a combinatorial procedure, called \emph{cutting}, which turns out to be the inverse process of the amalgamated sums considered in Section \ref{ssection:amalgam}. More precisely, these cuttings provide epimorphisms in $\amr$ which are retractions of the monomorphisms constructed from amalgamated sums of rooted cluster algebras.

	\subsection{Cutting along separating families of variables}
		\begin{defi}[Separating families]
			Let $\Sigma=(\x,\ex,B)$ be a seed.  If there exist a subset $\Delta \subset (\x \setminus \ex)$ and a partition $\x= \x_1 \sqcup \x_2 \sqcup \Delta$ such that, with respect to this partition, the exchange matrix $B$ is of the form
			$$B = \left[\begin{array}{c|c|c}
				B^1_{11} & 0 & B^1_{12}\\
				\hline
				0 & B^2_{11} & B^2_{12}\\
				\hline
				B^1_{21} & B^2_{21} & B_{\Delta}
			\end{array}\right],$$
			then we say that \emph{$\Delta$ separates $\x_1$ and $\x_2$ in $\Sigma$}.
		\end{defi}

		For $j \in \ens{1,2}$, we set $d_\Delta^j \Sigma =(\x_j \sqcup \Delta,\ex \cap \x_j, B^j)$ where
		$$B^j = \left[\begin{array}{ll}
			B^j_{11} & B^j_{12} \\
			B^j_{21} & B^j_{22}
		\end{array}\right].$$

		\begin{defi}[Cutting]
			Let $\Sigma = (\x,\ex,B)$ be a seed and let $\Delta$ be a separating family of variables as above. The \emph{cutting of $\Sigma$ along $\Delta$} is the pair
			$$d_{\Delta}\Sigma = (d^1_\Delta\Sigma, d^2_\Delta\Sigma).$$
		\end{defi}

		\begin{exmp}
			Consider for instance the matrix
			$$B = \left[\begin{array}{rr|rr|r}
				0 & 1 & 0 & 0  & -1 \\
				-1 & 0 & 0 & 0 & 1 \\
				\hline
				0 & 0 & 0 & 1 & -1 \\
				0 & 0 & -1 & 0 & 1\\
				\hline
				1 & -1 & 1 & -1 & 0
			\end{array}\right]$$
			corresponding to the quiver
			
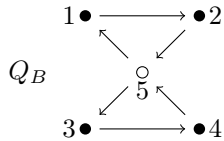
\begin{center}
				\begin{tikzpicture}[scale = .75]
					\fill (-1,-1) node {$Q_B$};

					\fill (0,0) circle (.1);
					\fill (2,0) circle (.1);

					\draw (1,-1) circle (.1);

					\fill (0,-2) circle (.1);
					\fill (2,-2) circle (.1);

					\draw[->] (.25,0) -- (1.75,0);
					\draw[<-] (.25,-.25) -- (.75,-.75);
					\draw[->] (1.75,-.25) -- (1.25,-.75);

					\draw[->] (.25,-2) -- (1.75,-2);
					\draw[<-] (.25,-1.75) -- (.75,-1.25);
					\draw[->] (1.75,-1.75) -- (1.25,-1.25);

					\fill (0,0) node [left] {1};
					\fill (0,-2) node [left] {3};
					\fill (2,0) node [right] {2};
					\fill (2,-2) node [right] {4};
					\fill (1,-1) node [below] {5};

				\end{tikzpicture}
			\end{center}
			with point 5 frozen. Then cutting along 5 gives two oriented 3-cycles with one frozen point each~:
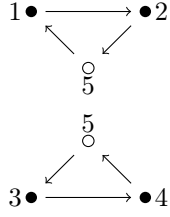
			\begin{center}
				\begin{tikzpicture}[scale = .75]
					\fill (0,0) circle (.1);
					\fill (2,0) circle (.1);

					\draw (1,-1) circle (.1);

					\draw[->] (.25,0) -- (1.75,0);
					\draw[<-] (.25,-.25) -- (.75,-.75);
					\draw[->] (1.75,-.25) -- (1.25,-.75);

					\fill (0,0) node [left] {1};
					\fill (2,0) node [right] {2};
					\fill (1,-1) node [below] {5};

				\end{tikzpicture}

				\begin{tikzpicture}[scale = .75]
					\draw (1,-2) circle (.1);

					\fill (0,-3) circle (.1);
					\fill (2,-3) circle (.1);

					\draw[->] (.25,-3) -- (1.75,-3);
					\draw[<-] (.25,-2.75) -- (.75,-2.25);
					\draw[->] (1.75,-2.75) -- (1.25,-2.25);

					\fill (0,-3) node [left] {3};
					\fill (2,-3) node [right] {4};
					\fill (1,-2) node [above] {5};
				\end{tikzpicture}
			\end{center}
		\end{exmp}

		The following lemmata prove that the cutting is the inverse operation to the amalgamated sum of seeds.
		\begin{lem}\label{lem:cutglue}
			Let  $\Sigma_1$ and $\Sigma_2$ be seeds which are glueable along $\Delta_1$ and $\Delta_2$ and let $\Sigma = \Sigma_1 \coprod_{\Delta_1,\Delta_2} \Sigma_2$. We denote by $\Delta$ the subset of the cluster of $\Sigma$ corresponding to $\Delta_1$ and $\Delta_2$. Then~:
			\begin{enumerate}
				\item $\Delta$ separates $\x_1 \setminus \Delta$ and $\x_2 \setminus \Delta$ in $\Sigma$~;
				\item $d^i_{\Delta}\Sigma \simeq \Sigma_i$ for any $i \in \ens{1,2}$.
			\end{enumerate}
		\end{lem}
		\begin{proof}
			By definition, we have $\Sigma = \Sigma_1 \coprod_{\Delta_1,\Delta_2} \Sigma_2 = (\x,\ex,B)$ where $\x= (\x_1 \setminus \Delta_1) \sqcup (\x_2 \setminus \Delta_2) \sqcup \Delta$, $\ex = \ex_1 \sqcup \ex_2$ and
			$$B = \left[\begin{array}{c|c|c}
				B^1_{11} & 0 & B^1_{12}\\
				\hline
				0 & B^2_{11} & B^2_{12}\\
				\hline
				B^1_{21} & B^2_{21} & B_{\Delta}
			\end{array}\right].$$
			Therefore, for any $i=1,2$, the cluster of $d_\Delta^i \Sigma$ is $\x_i \setminus \Delta_i \sqcup \Delta$, the exchangeable variables in this cluster are the exchangeable variables in $\Sigma_i$ (none of them belongs to $\Delta$ by assumption) and the exchange matrix of this seed is 
			$$B^i = \left[\begin{array}{c|c}
				B^i_{11} & B^1_{12}\\
				\hline
				B^i_{21} & B_{\Delta}
			\end{array}\right].$$
			Therefore, $d_\Delta^i \simeq \Sigma_i$.
		\end{proof}

		Conversely~:
		\begin{lem}\label{lem:gluecut}
			Let $\Sigma$ be a seed and $\Delta$ a separating family of variables in $\Sigma$. Then $d^1_\Delta \Sigma$ and $d^2_\Delta \Sigma$ are glueable along the respective images $\Delta_1$ and $\Delta_2$ of $\Delta$ and
			$$d^1_{\Delta} \Sigma \coprod_{\Delta_1,\Delta_2} d^2_{\Delta} \Sigma \simeq \Sigma.$$
		\end{lem}
		\begin{proof}
			We write $\Sigma=(\x,\ex,B)$. Since $\Delta$ is separating in $\Sigma$, there exists a partition $\x = \x_1 \sqcup \x_2 \sqcup \Delta$ such that, adapted to this partition, $B$ is given by
			$$B = \left[\begin{array}{c|c|c}
				B^1_{11} & 0 & B^1_{12}\\
				\hline
				0 & B^2_{11} & B^2_{12}\\
				\hline
				B^1_{21} & B^2_{21} & B_{\Delta}
			\end{array}\right]$$
			and since $\Delta$ consists of frozen variables, $\ex = (\ex \cap \x_1) \sqcup (\ex \cap \x_2)$.

			By definition, for $i=1,2$, we have $d^i_\Delta \Sigma = \Sigma_i$ where
			$$\Sigma_i = (\x_i \sqcup \Delta, \ex \cap \x_i, B^i)$$
			with
			$$B^i = \left[\begin{array}{c|c}
				B^i_{11} & B^1_{12}\\
				\hline
				B^i_{21} & B_{\Delta}
			\end{array}\right].$$
			It follows that $\Sigma_1$ and $\Sigma_2$ are glueable along $\Delta,\Delta$ and thus
			$$\Sigma_1 \coprod_{\Delta,\Delta} \Sigma_2 \simeq (\x, \ex, B) = \Sigma.$$
		\end{proof}

	\subsection{Epimorphisms from cuttings}
		Let $\Sigma$ be a seed and $\Delta$ a separating family of variables in $\Sigma$ as above and set $\Sigma_i = d^i_\Delta \Sigma$ for any $i \in \ens{1,2}$. 

		Fix $i \in \ens{1,2}$. It follows from Lemmata \ref{lem:cutglue}, \ref{lem:gluecut} and \ref{lem:injcoprod} that we have canonical monomorphisms in $\amr$~:
		$$j_i: \mathcal A(\Sigma_i) \fl \mathcal A(\Sigma).$$
		Consider the ring homomorphism~:
		$$\pr_i:\left\{\begin{array}{rcl}
			\mathcal F_{\Sigma} & \fl & \mathcal F_{\Sigma_i}\\
			x & \mapsto & x \text{ if } x \in \x_i \sqcup \Delta,\\
			x & \mapsto & 0 \text{ if } x \in \x_j \text{ for } i\neq j.
		\end{array}\right.$$

		\begin{prop}
			For any $i \in \ens{1,2}$, the ring homomorphism $\pr_i$ induces a rooted cluster epimorphism $\mathcal A(\Sigma) \fl \mathcal A(\Sigma_i)$ which is a retraction for $j_i$.
		\end{prop}
		\begin{proof}
			We first observe that 
			$$\pr_i(\x) \subset \x_i \sqcup \Delta  \sqcup \ens 0 \text{ and }\pr_i(\ex) \subset (\ex \cap \x_i) \sqcup \ens 0$$
			so that $\pr_i$ satisfies \AM 1 and \AM 2.

			In order to prove that $\pr_i$ satisfies \AM 3, we prove as in Lemma \ref{lem:injcoprod} that the $(\pr_i,\Sigma,\Sigma_i)$-biadmissible sequences are precisely the $\Sigma_i$-admissible sequences and that $\Delta$ remains separating along biadmissible mutations. It follows as in the proof of Lemma \ref{lem:injcoprod} that $\pr_i$ commutes with biadmissible mutations and thus satisfies \AM 3 and induces a surjective rooted cluster morphism $\mathcal A(\Sigma) \fl \mathcal A(\Sigma_i)$. 

			Finally, as $(\pr_i \circ j_i)(x) = x$ for any $x \in \x_i \sqcup \Delta$, we get $\pr_i \circ j_i = \1_{\mathcal A(\Sigma_i)}$.
		\end{proof}

		In general, we state the following problem~:
		\begin{prob}
			Determine which monomorphisms in $\am$ are sections.
		\end{prob}

	\subsection{Topological interpretation of the surgery}\label{ssection:topsurgery}
		Let $(S,M)$ be a marked surface. Assume that there exists a collection $\Delta$ of (internal or boundary) arcs in $(S,M)$ which can be concatenated in order to form a simple closed curve in $(S,M)$ which delimits two non-degenerate marked subsurfaces $(S_1,M_1)$ and $(S_2,M_2)$ in $(S,M)$. 

		Consider a triangulation $T$ of $(S,M)$ containing $\Delta$ as a subset. Then $T$ induces two triangulations $T_1$ and $T_2$ of $(S_1,M_1)$ and $(S_2,M_2)$ respectively in which $\Delta$ corresponds to a set of boundary arcs, that is, to frozen variables. We denote by $\Sigma$ the seed corresponding to the triangulation $T$ and by $\Sigma_1$, $\Sigma_2$ the seeds corresponding respectively to the triangulations $T_1$ and $T_2$. Then $\Delta$ is a separating family of variables in $\Sigma$ and 
		$$d_\Delta(\Sigma) = (\Sigma_1,\Sigma_2).$$
		In other words, cutting $\Sigma$ along $\Delta$ coincides with taking the seed associated with the triangulation of the surface obtained by cutting $(S,M)$ along $\Delta$.

		\begin{exmp}
			Consider the following triangulation of the disc with two marked points on the boundary and five punctures.
		
			
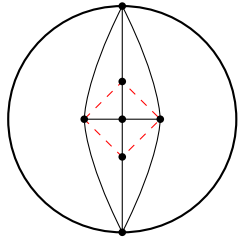
\begin{center}
				\begin{tikzpicture}[scale = .5]
					\draw[thick] (0,0) circle (3);

					\draw[dashed,red] (-1,0) -- (0,1) -- (1,0) -- (0,-1) -- cycle;

					\fill (0,3) circle (.1);
					\fill (0,-3) circle (.1);

					\fill (0,1) circle (.1);
					\fill (0,-1) circle (.1);
					\fill (1,0) circle (.1);
					\fill (-1,0) circle (.1);

					\fill (0,0) circle (.1);

					\draw (0,0) -- (0,1);
					\draw (0,0) -- (0,-1);
					\draw (0,0) -- (1,0);
					\draw (0,0) -- (-1,0);

					\draw (0,3) -- (0,1);
					\draw (0,3) .. controls (1,1) and (1,0) .. (1,0);
					\draw (0,3) .. controls (-1,1) and (-1,0) .. (-1,0);

					\draw (0,-3) -- (0,-1);
					\draw (0,-3) .. controls (1,-1) and (1,0) .. (1,0);
					\draw (0,-3) .. controls (-1,-1) and (-1,0) .. (-1,0);
				\end{tikzpicture}
			\end{center}

			Let $\Delta$ denote the union of the four arcs which are dashed in the above figure. Cutting the surface along $\Delta$ gives two new marked surfaces, namely an unpunctured annulus with two marked points on a boundary and four marked points on the other and a disc with four marked points on the boundary and one puncture. The triangulation of the above disc containing $\Delta$ thus induces triangulations of the two cut surfaces.

			
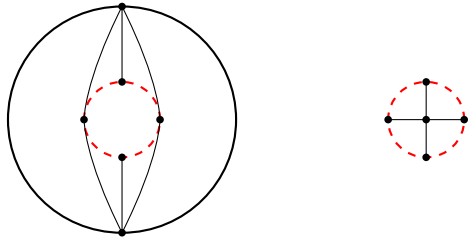
\begin{center}
				\begin{tikzpicture}[scale = .5]
					\draw[thick] (0,0) circle (3);

					\draw[thick,dashed,red] (0,0) circle (1);

					\fill (0,3) circle (.1);
					\fill (0,-3) circle (.1);

					\fill (0,1) circle (.1);
					\fill (0,-1) circle (.1);
					\fill (1,0) circle (.1);
					\fill (-1,0) circle (.1);

					\draw (0,3) -- (0,1);
					\draw (0,3) .. controls (1,1) and (1,0) .. (1,0);
					\draw (0,3) .. controls (-1,1) and (-1,0) .. (-1,0);

					\draw (0,-3) -- (0,-1);
					\draw (0,-3) .. controls (1,-1) and (1,0) .. (1,0);
					\draw (0,-3) .. controls (-1,-1) and (-1,0) .. (-1,0);

			% % % % % % % % DEUXIEME SURFACE

					\draw[thick,dashed,red] (8,0) circle (1);

					\fill (8,1) circle (.1);
					\fill (8,-1) circle (.1);
					\fill (9,0) circle (.1);
					\fill (7,0) circle (.1);

					\fill (8,0) circle (.1);

					\draw (8,0) -- (8,1);
					\draw (8,0) -- (8,-1);
					\draw (8,0) -- (9,0);
					\draw (8,0) -- (7,0);

				\end{tikzpicture}
			\end{center}

			Then we clearly see that the cluster algebra associated with the initial surface where the arcs in $\Delta$ are frozen is the amalgamated sum over the respective images of $\Delta$ of the two cut surfaces.
		\end{exmp}

\section*{Acknowledgements}
	The first author gratefully acknowledges partial support from NSERC of Canada, FQRNT of Qu\'ebec and the University of Sherbrooke.

	This work was initiated while the second author was a CRM-ISM postdoctoral fellow at the University of Sherbrooke and was finished while he was a postdoctoral fellow at the Universit\'e Denis Diderot -- Paris 7, funded by the \emph{ANR Geom\'etrie Tropicale et Alg\`ebres Amass\'ees}. The second author would like to thank the third author for his kind hospitality during his stay in University of Connecticut where this project started.

	The third author has been supported by the NSF Grant 1001637 and by the University of Connecticut.

	The authors wish to thank Bernhard Keller, Yann Palu, Pierre-Guy Plamondon, Jan Schr\"oer and Vasilisa Shramchenko for fruitful discussions on the topic.

\newcommand{\etalchar}[1]{$^{#1}$}
\providecommand{\bysame}{\leavevmode\hbox to3em{\hrulefill}\thinspace}
\providecommand{\MR}{\relax\ifhmode\unskip\space\fi MR }
\providecommand{\MRhref}[2]{%
  \href{http://www.ams.org/mathscinet-getitem?mr=#1}{#2}
}
\providecommand{\href}[2]{#2}

\end{document}